\documentclass[12pt,reqno]{amsart}
\usepackage{amsmath}
\usepackage{amssymb}
\usepackage{amstext}
\usepackage{mathrsfs}
\usepackage{a4wide}
\usepackage{graphicx}
\usepackage{bbm}

\allowdisplaybreaks \numberwithin{equation}{section}
\usepackage{color}
\usepackage{cases}

\numberwithin{equation}{section}

\newtheorem{theorem}{Theorem}[section]

\newtheorem{lemma}[theorem]{Lemma}

\theoremstyle{definition}

\theoremstyle{remark}
\newtheorem{remark}[theorem]{Remark}

\begin{document}

\title
[K\'arm\'an vortex street for the gSQG equation]{K\'arm\'an vortex street for the generalized surface quasi-geostrophic equation}

\author{Daomin Cao, Guolin Qin,  Weicheng Zhan, Changjun Zou}

\address{Institute of Applied Mathematics, Chinese Academy of Sciences, Beijing 100190, and University of Chinese Academy of Sciences, Beijing 100049,  P.R. China}
\email{dmcao@amt.ac.cn}
\address{Institute of Applied Mathematics, Chinese Academy of Sciences, Beijing 100190, and University of Chinese Academy of Sciences, Beijing 100049,  P.R. China}
\email{qinguolin18@mails.ucas.edu.cn}
\address{Institute of Applied Mathematics, Chinese Academy of Sciences, Beijing 100190, and University of Chinese Academy of Sciences, Beijing 100049,  P.R. China}
\email{zhanweicheng16@mails.ucas.ac.cn}
\address{Institute of Applied Mathematics, Chinese Academy of Sciences, Beijing 100190, and University of Chinese Academy of Sciences, Beijing 100049,  P.R. China}
\email{zouchangjun17@mails.ucas.ac.cn}

\thanks{This work was supported by NNSF of China Grant 11831009 and Chinese Academy of Sciences (No. QYZDJ-SSW-SYS021).}

\begin{abstract}
	We are concerned with the existence of periodic travelling-wave solutions for the generalized surface quasi-geostrophic (gSQG) equation(including incompressible Euler equation), known as von K\'arm\'an vortex street. These solutions are of $C^1$ type, and are obtained by studying a semilinear problem on an infinite strip whose width equals to the period. By a variational characterization of solutions, we also show the relationship between vortex size, travelling speed and street structure. In particular, the vortices with positive and negative intensity have the same or different scaling size in our construction, which constitutes the regularization for K\'arm\'an point vortex street.
\end{abstract}

\maketitle

{\small{\bf Keywords:} K\'arm\'an vortex street; the gSQG equation; $C^1$ type solutions; Lyapunov-Schmidlt reduction. }\\

\section{Introduction and main results}
When a two-dimensional bluff body is placed in a uniform stream moving at certain velocities, vortices with opposite intensity will arise along two parallel staggered rows, which is observed as water flow going through a pipe, or wind passing an obstacle. The best-known event caused by this pattern is the fall of Tacoma narrows bridge in 1940. Experimental study of periodic vortex shedding can be traced back to 1870s in \cite{Ray,Sto}, while the theoretical model was proposed by von K\'arm\'an \cite{Kar1,Kar2}, and hence this phenomenon is known as von K\'arm\'an vortex street nowadays in literatures. For the reason that the exact problem is complex from a theoretical point of view, some simplified models were investigated in \cite{Are,Lamb,Roh}. The main idea is using different kinds of solutions to approximate K\'arm\'an point vortex street, since the latter is the basic and simplest pattern of periodic vortex shedding.

It is notable that although viscosity and bluff body are involved in the generation of K\'arm\'an vortex street, they seem not to influence anymore the evolution of the vortex street(For more details on the effect of Reynold number and shape of bluff body, we refer to \cite{Gac2,Jim,Mat} and references therein). This fact indicates that an inviscid incompressible fluid model can be used to describe the vortex dynamics in K\'arm\'an vortex street. In \cite{Saf1,Saf2}, Saffman and Schatzman studied K\'arm\'an vortex street for Euler flow. Under the assumption that the support of each vortex has finite area, they conducted a series of numerical simulation to show the existence of one-directional periodic vortex shedding travelling at a constant speed. Moreover, they obtained a linear stability of K\'arm\'an vortex street where the size of vortices and street width satisfy a special condition.

We are going to study the existence of $C^1$ type K\'arm\'an vortex street for the generalized surface quasi-geostrophic (gSQG) equation, which can be written as follows
\begin{align}\label{1-1}
	\begin{cases}
		\partial_t\vartheta+\mathbf{v}\cdot \nabla \vartheta =0&\text{in}\ \mathbb{R}^2\times (0,T),\\
		\ \mathbf{v}=\nabla^\perp\psi, \ \psi=(-\Delta)^{-s}\vartheta     &\text{in}\ \mathbb{R}^2\times (0,T),\\
		\vartheta\big|_{t=0}=\vartheta_0 &\text{in}\ \mathbb{R}^2,\\
	\end{cases}
\end{align}
with $ 0<s\le1$, where $(x_1,x_2)^\perp=(x_2,-x_1)$, $\nabla^\perp=(\frac{\partial}{\partial x_2}, -\frac{\partial}{\partial x_1})$,  $\vartheta(x,t):\mathbb{R}^2\times (0,T)\to \mathbb{R}$ the active scalar being transported by the velocity field $\mathbf{v}( x,t):\mathbb{R}^2\times (0,T)\to \mathbb{R}^2$ generated by $\vartheta$, and $\psi$ the corresponding stream function. The operator $(-\Delta)^{-s}$ is defined by
\begin{equation*}
(-\Delta)^{-s}\vartheta(x)=\int_{\mathbb{R}^2}G_s(x- y)\vartheta(y)d y,
\end{equation*}
where $G_s$ is the fundamental solution of $(-\Delta)^{-s}$ in $\mathbb{R}^2$ given by
\begin{equation*}
	G_s(x)=\left\{
	\begin{array}{lll}
		\frac{1}{2\pi}\ln \frac{1}{|x|},  \ \ \ \ \ \ \ \ \ \ & \text{if} \ \ s=1,\\
		\frac{c_s}{|x|^{2-2s}}, \ \ \ c_s=\frac{\Gamma(1-s)}{2^{2s}\pi\Gamma(s)}, & \text{if} \ \ 0< s<1,
	\end{array}
	\right.
\end{equation*}
with $\Gamma$ the Euler gamma function.

When $s=1$, \eqref{1-1} is the vorticity formulation of 2D incompressible Euler equation. When $s=\frac{1}{2}$, \eqref{1-1} is the surface quasi-geostrophic (SQG) equation, which is relevant to the atmosphere circulation and ocean dynamics \cite{Con}. The gSQG model \eqref{1-1} with $0<s<1$ was proposed by C\'{o}rdoba et al. in \cite{Cor}, and was taken as a generalization of the Euler equation and the SQG equation.

In 1963, Yudovich \cite{Yud} proved the global well-posedness of \eqref{1-1} with the initial data in $L^1\cap L^\infty$ for $s=1$. However, the global well-posedness for the general case $0<s<1$ remains unknown due to the loss of regularity for velocity field. In \cite{Con}, Constantin et al.  established local well-posedness of the gSQG equation for classical solutions, which is known for sufficiently regular initial data by \cite{Chae,Gan,Kis2}. The study of local existence in different function spaces can be found in \cite{Chae0,Li0,Wu1,Wu2}. Resnick \cite{Res} proved global existence for weak solutions to the SQG equations with any initial data in $L^2$. This remarkable result was then improved by Marchand \cite{Mar} to any initial data belonging to $L^p$ with $p>4/3$. On the other hand, Kiselev and Nazarov \cite{Kis} constructed solutions of the gSQG equations with arbitrary Sobolev growth.

As concrete examples for the gSQG flow, various kinds of global solutions to \eqref{1-1} are constructed. There are mainly two kinds of global solutions: The rotating solutions and the travelling-wave solutions. The rotating solutions are also known as the V-states, and the first explicit non-trivial V-state is Kirchhoff ellipse given in \cite{Kir} for $s=1$. In the past decades, different methods were developed to construct solutions of this type, and we refer to \cite{Ao,Cas1,Has,HM,T2} for more discussion. As for the travelling-wave solutions, the early example is the Lamb dipole or Chaplygin-Lamb dipole \cite{Lamb}, which is a travelling vortex pair in the case $s=1$. In \cite{Ao,Cao4,Go0,Gra}, several kinds of travelling-wave solutions were given by a similar approach for rotating solutions. We shall bring the attention of readers to that K\'arm\'an vortex street is a special kind of travelling-wave solution other than vortex pairs, which consists of infinite vortices and has a periodic structure.
Furthermore, different from vortex pairs which can only travel along their axis of symmetry, the uniform travelling speed of a vortex street can be chosen in other directions by adjusting the phase difference of two sides of the street. We will show these properties later in our main theorem.

To explain the problem we are to address and state our results, we need to introduce notations for convenience: $\boldsymbol\delta_z$ is the Dirac measure located at $z\in \mathbb R^2$, $\boldsymbol\chi_\Omega$ denotes the characteristic function of  $\Omega\subset\mathbb R^2$, $\boldsymbol e_i$ is the unit vector of $x_i$ axis for $i=1,2$; $O_\varepsilon(1)$ will be used to denote quantities which stay bounded as $\varepsilon$ goes to zero and $o_\varepsilon(1)$ to denote quantities which go to zero as $\varepsilon$ goes to zero. $O_\varepsilon(1)$ and $o_\varepsilon(1)$ only depend on $\varepsilon$.

As a preliminary, we cast an eye on the most singular type of K\'arm\'an vortex street, where solutions to \eqref{1-1} are composed of two parallel rows of point vortices. If we denote
\begin{equation*}
	p=(-d,-a), \ \ \ \ \ q=(d,a)
\end{equation*}
with $d\ge0$ as half of street width and $a\in [0,l/2)$ as half of the phase(The case $d=a=0$ must be ruled out), then these $x_2$-directional periodic travelling-wave solutions take the form
\begin{equation*}
	\vartheta(x,t)=\sum\limits_{k\in \mathbb Z}\boldsymbol\delta_p(x+kl\boldsymbol e_2-t\mathbf U_{d,l,a})-\sum\limits_{k\in \mathbb Z}\boldsymbol\delta_q(x+kl\boldsymbol e_2-t\mathbf U_{d,l,a}),
\end{equation*}
where $l>0$ is the period length, and $\mathbf U_{d,l,a}\in \mathbb{R}^2$ is the uniform travelling speed. According to the dynamic formula for point vortex model given by Rosenzweig \cite{Ros}, the travelling speed $\mathbf U_{d,l,a}$ can be computed directly as
\begin{equation*}
	\mathbf U_{d,l,a}=-\mathcal C_s\lim\limits_{N\to\infty}\sum\limits_{|k|\le N }\frac{(p-q+kl\boldsymbol e_2)^\perp}{|p-q+kl\boldsymbol e_2|^{4-2s}},
\end{equation*}
where
\begin{equation}\label{1-2}
	\mathcal C_s=\left\{
	\begin{array}{lll}
		1/2\pi  \ \ \ \ \ \ \ \ \ \ & \text{if} \ \ s=1,\\
		(2-2s)c_s, & \text{if} \ \ 0<s<1.
	\end{array}
	\right.
\end{equation}
In particular, when $a=0$ or $l/4$ we can use symmetry of the solution with respect to $x_2=kl$ or $kl\pm l/4$ to derive that
\begin{equation*}
\mathbf U_{d,l,0}=W_1\boldsymbol e_2,  \ \ \ \ \ \mathbf U_{d,l,l/4}=W_2\boldsymbol e_2,
\end{equation*}
where
\begin{equation}\label{1-3}
	W_1(d)=\mathcal C_s\lim\limits_{N\to\infty}\sum\limits_{|k|\le N}\frac{2d}{(4d^2+k^2l^2)^{2-s}},
\end{equation}
and
\begin{equation}\label{1-4}
	W_2(d)=\mathcal C_s\lim\limits_{N\to\infty}\sum\limits_{|k|\le N}\frac{2d}{(4d^2+(kl+\frac{l}{2})^2)^{2-s}}.
\end{equation}

Recently, Garc\'ia \cite{Gac2} constructed a family of patch type solutions to approximate K\'arm\'an point vortex street with $1/2<s\le 1$. These solutions have the following explicit expression
\begin{equation*}
	\vartheta_\varepsilon(x)=\frac{1}{\varepsilon^2\pi}\sum\limits_{k\in \mathbb Z}\boldsymbol\chi_{D_\varepsilon+(-d,kl-a)}( x-t\mathbf U_\varepsilon)-\frac{1}{\varepsilon^2\pi}\sum\limits_{k\in \mathbb Z}\boldsymbol\chi_{-D_\varepsilon+(d,kl+a)}(x-t\mathbf U_\varepsilon),
\end{equation*}
where $D_\varepsilon$ is a perturbation of the disc $B_\varepsilon(0)$ centered at the origin with sufficiently small radius $\varepsilon$, and $\mathbf U_\varepsilon$ is the uniform travelling speed. The approach in \cite{Gac2} highly relies on the patch structure: Employing Biot-Savart law, the author obtained the contour dynamic equation for vortex boundary, and calculated its linearization at point vortex solutions. The key point of the construction is to choose $\mathbf U_\varepsilon=\mathbf U_{d,l,a}+o_\varepsilon(1)$ properly, so that an isomorphism condition is satisfied for linearized operator. Then a family of nontrivial solutions can be obtained by implicit function theorem.

In the present paper, we will focus on the construction of $C^1$ type K\'arm\'an vortex street. To be more precise, we will prove the existence of travelling-wave solutions to \eqref{1-1} with the formulation
\begin{equation}\label{1-5}
	\vartheta_\varepsilon(x,t)=\vartheta_{0,\varepsilon}(x-t\mathbf U_\varepsilon),
\end{equation}
where $\mathbf U_\varepsilon\in \mathbb R^2$ is the uniform travelling speed, $\varepsilon$ is some size parameter, and the initial data $\vartheta_{0, \varepsilon}(x)$ is given by
\begin{equation}\label{1-6}
	\vartheta_{0,\varepsilon}(x)=\sum\limits_{k\in \mathbb Z}\vartheta_{1,\varepsilon}(x+kl\boldsymbol e_2)-\sum\limits_{k\in \mathbb Z}\vartheta_{2,\varepsilon}(x+kl\boldsymbol e_2).
\end{equation}
Here, we assume that $\vartheta_{1,\varepsilon}(x),\vartheta_{2,\varepsilon}(x)\in C^1(\mathbb R^2)$ and satisfy
\begin{equation*}
supp(\vartheta_{1,\varepsilon})\subset B_{L\varepsilon}((-d,-a)), \ \ \ supp(\vartheta_{2,\varepsilon})\subset B_{L\sigma(\varepsilon)}((d,a)),
\end{equation*}
where $L>0$ is some large constant, $\sigma(\varepsilon)>0$ is the size function in the sense that $L\varepsilon $ and $L\sigma(\varepsilon)$ are the upper bounds  for the diameters of supports of vortices with positive vorticity and vortices with negative vorticity respectively. A novelty of our construction is that vortices on the right hand side may have a different size function compared with those on the left hand side, that is, we make the following assumption on $\sigma(\varepsilon)$:
\begin{itemize}
	\item[\textbf{(H)}] As $\varepsilon\to 0$, $\sigma(\varepsilon)/\varepsilon\le C$ for fixed $C>0$, and $\varepsilon^\tau/\sigma(\varepsilon)=o_\varepsilon(1)$ for some $1<\tau\le 2$.
\end{itemize}

There are several difficulties in the construction of $C^1$ type solutions mentioned above. Firstly, we do not impose any symmetry with respect to $x_2$-axis, and $\vartheta_{1,\varepsilon}(x)$, $\vartheta_{2,\varepsilon}(x)$ may have different profiles apart from the difference in vortex size. Secondly, due to the general $C^1$ type vorticity, the velocity of flow can not be recovered by vortex boundary alone, and the method by studying contour dynamic equation is invalid. To achieve our goal, we will take another approach, which is from a new angle of view but also reduces the construction into a finite-dimensional problem. We will briefly explain our strategy. For easy understanding, we first assume $a=0$, $\vartheta_{0,\varepsilon}$ is symmetric with respect to $x_2=kl$; or $a=l/4$, $\vartheta_{0,\varepsilon}$ is symmetric with respect to $x_2=kl\pm l/4$, so that the travelling speed is in $x_2$ direction and we can write $\mathbf U_\varepsilon=W_\varepsilon\boldsymbol e_2$ for some scalar $W_\varepsilon$.

According to \eqref{1-5}, by introducing the $x_2$-directional periodic stream function $\tilde\psi_\varepsilon$, \eqref{1-1} can be rewritten as
\begin{equation}\label{1-7}
	(\nabla^\perp \tilde\psi_\varepsilon-W_\varepsilon\boldsymbol e_2)\cdot \nabla \vartheta_{0,\varepsilon}(x)=0, \ \ \ \tilde\psi_\varepsilon=(-\Delta)^{-s}\vartheta_{0,\varepsilon} \ \ \ \ \ \text{in} \ \ \mathbb R^2,
\end{equation}
which means $\vartheta_{0,\varepsilon}$ is functional related to $\tilde\psi_\varepsilon+W_\varepsilon x_1$. It is natural to impose
$\vartheta_{0,\varepsilon}=f(\tilde\psi_\varepsilon+W_\varepsilon x_1)$ for some $C^1$ monotone $f$, and transform \eqref{1-7} into a semilinear elliptic equation
\begin{equation}\label{1-8}
	(-\Delta)^s \tilde\psi_\varepsilon=f(\tilde\psi_\varepsilon+W_\varepsilon x_1) \ \ \ \ \ \ \ \text{in} \ \ \mathbb R^2.
\end{equation}
One can easily verify that \eqref{1-8} provides a family of classical solutions to \eqref{1-7} by theory of regularity for elliptic equations. We will follow the framework in \cite{Ao, Cao5} to construct desired solutions to \eqref{1-8} by a Lyapunov-Schmidlt reduction.

However, there are several new ideas in our construction: since $\tilde\psi_\varepsilon$ and $\vartheta_{0,\varepsilon}$ are periodic over $\mathbb R^2$, the energy of K\'arm\'an vortex street is infinite, which leads to a difficulty for variational characterization of solutions. Inspired by \cite{Bar} on one-dimensional periodic problem, we will study \eqref{1-8} restricted in an infinite strip whose width equals the $x_2$ direction period $l$, namely
\begin{equation}\label{1-9}
	(-\Delta)_*^s \psi_\varepsilon=f(\psi_\varepsilon+Wx_1) \ \ \ \ \ \ \ \text{in} \ \ \mathbb R\times (-l/2,l/2),
\end{equation}
where $\psi_\varepsilon$ is $\tilde\psi_\varepsilon$ restricted in $\mathbb R\times (-l/2,l/2)$, and $(-\Delta)_*^s$ is  $(-\Delta)^s$ restricted in the corresponding typical period. We will give the explicit formula for $(-\Delta)_*^s$ in Section 2. Moreover, our construction needs much more careful estimate compared with \cite{Ao} due to the different sizes of positive and negative vortices.

When the solvability of \eqref{1-9} is considered, another problem arises from the fundamental solution of $(-\Delta)^s$: $G_s(x)$ is of order $|x|^{2s-2}$ when $0<s<1$, and $\ln (1/|x|)$ when $s=1$. But $2-2s$ is unsatisfactorily less than $1$ if $s>1/2$. This fact may cause the divergence for $L^\infty$ norm of $\psi_\varepsilon$ when we deal with the influence of infinite vortices. Thanks to the unique structure of K\'arm\'an vortex street, where each positive vortex matches a negative vortex with equal intensity, we observe that the influence of two equally distant vortex pairs is actually of order $|x|^{2s-4}$. As a result, $\psi_\varepsilon$ has a convergent $L^\infty$ norm, and our method does work as desired.

Having made the preparations, we are now in the position to state our first result.
\begin{theorem}\label{thm1}
	Suppose $s\in (0,1]$. Then there exist $\varepsilon_0>0$ and $\tau$ in \textbf{(H)} such that for any $\varepsilon\in (0,\varepsilon_0)$, \eqref{1-1} has a $x_2$-directional periodic travelling-wave solution $\vartheta_\varepsilon(x,t)=\vartheta_{0,\varepsilon}(x-tW_\varepsilon\boldsymbol e_2)$, where the initial data $\vartheta_{0,\varepsilon}(x)\in C^1(\mathbb{R}^2)$  is symmetric with respect to $x_2=kl$ for $k\in\mathbb Z$, $l>0$, and  has the form
	\begin{equation*}
		\vartheta_{0,\varepsilon}(x)=\sum\limits_{k\in \mathbb Z}\vartheta_{1,\varepsilon}(x+kl\boldsymbol e_2)-\sum\limits_{k\in \mathbb Z}\vartheta_{2,\varepsilon}(x+kl\boldsymbol e_2),	
	\end{equation*}
	with $l>0$, $supp(\vartheta_{1,\varepsilon})\subset B_{L\varepsilon}((-d,0))$, $supp(\vartheta_{2,\varepsilon})\subset B_{L\sigma(\varepsilon)}((d,0))$ for $\sigma(\varepsilon)$ satisfying \textbf{(H)}, $d>0$, and some large $L>0$. The scalar $W_\varepsilon$ satisfies
	\begin{equation*}
		W_\varepsilon=\mathcal C_s\lim\limits_{N\to\infty}\sum\limits_{|k|\le N}\frac{2d}{(4d^2+k^2l^2)^{2-s}}+o_\varepsilon(1),
	\end{equation*}
	where $\mathcal C_s$ is given in \eqref{1-2}. Moreover, it holds in the sense of measure
	\begin{equation*}
		\vartheta_{0,\varepsilon}(x)\rightharpoonup \sum\limits_{k\in \mathbb Z}\boldsymbol\delta_{(-d,kl)}(x)-\sum\limits_{k\in \mathbb Z}\boldsymbol\delta_{(d,kl)}(x) \ \ \ \text{as} \ \ \varepsilon\to 0,
	\end{equation*}
\end{theorem}
\begin{remark}
	When $0<s<1$, the function $W_1(d)$ given in \eqref{1-3} is monotonically decreasing whose range is $(0,+\infty)$. As a result, $W_\varepsilon$ can take any positive values by adjusting $d$. When $s=1$, there is an explicit formula $W_1(d)=\frac{1}{2l}\coth (\frac{\pi d}{2l})$. Since the range of $\coth$ on $\mathbb R_+$ is $(1,\infty)$, we deduce that $W_\varepsilon>1/4l$ provided $\varepsilon$ is sufficiently small.
\end{remark}

From \eqref{1-3}, we see that Theorem \ref{thm1} corresponds to the K\'arm\'an point vortex street for $a=0$ when we let $\varepsilon\to 0$. As a counterpart of \eqref{1-4}, the result for the case $a=l/4$ can be stated as follows.
\begin{theorem}\label{thm2}
	Suppose $s\in (0,1]$. Then there exist $\varepsilon_0>0$ and $\tau$ in \textbf{(H)} such that for any $\varepsilon\in (0,\varepsilon_0)$, \eqref{1-1} has a $x_2$-directional periodic travelling-wave solution $\vartheta_\varepsilon(x,t)=\vartheta_{0,\varepsilon}(x-tW_\varepsilon\boldsymbol e_2)$, where the initial data $\vartheta_{0,\varepsilon}(x)\in C^1(\mathbb{R}^2)$ is symmetric with respect to $x_2=kl\pm l/4$ for $k\in\mathbb Z$, $l>0$, and has the form
	\begin{equation*}
		\vartheta_{0,\varepsilon}(x)=\sum\limits_{k\in \mathbb Z}\vartheta_{1,\varepsilon}(x+kl\boldsymbol e_2)-\sum\limits_{k\in \mathbb Z}\vartheta_{2,\varepsilon}(x+kl\boldsymbol e_2).	
	\end{equation*}
	with $supp(\vartheta_{1,\varepsilon})\subset B_{L\varepsilon}((-d,-l/4))$, $supp(\vartheta_{2,\varepsilon})\subset B_{L\sigma(\varepsilon)}((d,l/4))$ for $\sigma(\varepsilon)$ satisfying \textbf{(H)}, $d\ge 0$, and some large $L>0$. The scalar $W_\varepsilon$ satisfies
	\begin{equation*}
		W_\varepsilon=\mathcal C_s\lim\limits_{N\to\infty}\sum\limits_{|k|\le N}\frac{2d}{(4{d}^2+(kl+\frac{l}{2})^2)^{2-s}}+o_\varepsilon(1)
	\end{equation*}
	with $\mathcal C_s$ given in \eqref{1-2}. Moreover, it holds in the sense of measure
	\begin{equation*}
		\vartheta_{0,\varepsilon}(x)\rightharpoonup \sum\limits_{k\in \mathbb Z}\boldsymbol\delta_{(-d,kl-l/4)}(x)-\sum\limits_{k\in \mathbb Z}\boldsymbol\delta_{(d,kl+l/4)}(x) \ \ \ \text{as} \ \ \varepsilon\to 0.
	\end{equation*}
\end{theorem}
\begin{remark}
	When $0<s<1$, $W_2(d)$ in \eqref{1-4} will first increase and then decrease to $0$ on $\mathbb R_+$. Hence we have $W_\varepsilon<\sup_{\mathbb{R^+}}W_2+1$ if $\varepsilon$ is sufficiently small. While for $s=1$, it holds $W_1(d)=\frac{1}{2l}\tanh (\frac{\pi d}{2l})$. Since the range of $\tanh$ on $\mathbb R_+\cup\{0\}$ is $[0,1)$, we deduce that $W_\varepsilon<1/l$ as long as $\varepsilon$ is sufficiently small.
	
	Notice that $d$ can be $0$ in Theorem \ref{thm2}. In this special case, the vortex street is located along $x_2$-axis and nearly stagnating, namely, travelling speed is almost zero. In paticular, if we assume $\vartheta_{0,\varepsilon}$ is even in $x_1$-direction, then the solution $\vartheta_\varepsilon$ is stationary, which gives another example for nontrivial stationary solution to \eqref{1-1} with $0<s\le 1$ besides the one constructed in \cite{Gom}.
\end{remark}

More generally, we have the following result for arbitrary phase $a\in (0,l/2)$, where the uniform travelling speed $\mathbf U_\varepsilon$ can have different directions other than $x_2$-direction.
\begin{theorem}\label{thm3}
	Suppose $s\in (0,1]$, $p=(-d,-a)$ and $q=(d,a)$ with $d\ge0$, $a\in(0,l/2)$.  Then there exist $\varepsilon_0>0$ and $\tau$ in \textbf{(H)} such that for any $\varepsilon\in (0,\varepsilon_0)$, \eqref{1-1} has a $x_2$-directional periodic travelling-wave solution $\vartheta_\varepsilon(x,t)=\vartheta_{0,\varepsilon}(x-t\mathbf U_\varepsilon)$, where the initial data $\vartheta_{0,\varepsilon}(x)\in C^1(\mathbb{R}^2)$ has the form
	\begin{equation*}
		\vartheta_{0,\varepsilon}(x)=\sum\limits_{k\in \mathbb Z}\vartheta_{1,\varepsilon}(x+kl\boldsymbol e_2)-\sum\limits_{k\in \mathbb Z}\vartheta_{2,\varepsilon}(x+kl\boldsymbol e_2).	
	\end{equation*}
	with $l>0$, $supp(\vartheta_{1,\varepsilon})\subset B_{L\varepsilon}(p)$, $supp(\vartheta_{2,\varepsilon})\subset B_{L\sigma(\varepsilon)}(q)$ for $\sigma(\varepsilon)$ satisfying \textbf{(H)} and some large $L>0$. The uniform travelling speed $\mathbf U_\varepsilon\in \mathbb R^2$ satisfies
	\begin{equation*}
		\mathbf U_\varepsilon=-\mathcal C_s\lim\limits_{N\to\infty}\sum\limits_{|k|\le N }\frac{(p-q+kl\boldsymbol e_2)^\perp}{|p-q+kl\boldsymbol e_2|^{4-2s}}+o_\varepsilon(1)
	\end{equation*}
	with $\mathcal C_s$ given in \eqref{1-2}. Moreover, it holds in the sense of measure
	\begin{equation*}
		\vartheta_{0,\varepsilon}(x)\rightharpoonup \sum\limits_{k\in \mathbb Z}\boldsymbol\delta_{p}(x+kl\boldsymbol e_2)-\sum\limits_{k\in \mathbb Z}\boldsymbol\delta_{q}(x+kl\boldsymbol e_2) \ \ \ \text{as} \ \ \varepsilon\to 0.
	\end{equation*}
\end{theorem}

The solutions constructed above actually constitute the regularization for K\'arm\'an point vortex street. Recall that a vortex dynamic system is called a vortex-wave system, if it is composed of highly concentrated vortices known as ``vortex", and relatively scattered vortices known as ``wave". Suppose $\sigma(\varepsilon)$ satisfies \textbf{(H)} with $\sigma(\varepsilon)=o_\varepsilon(1)$. Then compared with the positive vortices, the size of negative vortices in Theorem \ref{thm1} \ref{thm2} and \ref{thm3} has a sharper shrinking rate. So, in this situation, our result can be regarded as the regularization of foresaid vortex-wave system with ``vortex" on the right and ``wave" on the left.

Our proof will begin with the relatively simple case $a=0$ or $l/4$. In Section 2, we consider the gSQG equation with $0<s<1$. We first use the periodic setting to construct a series of approximate solutions to \eqref{1-6}, and compute the error of approximation. Then we study the linear projective problem and make essential a priori estimate. The existence and uniqueness of solutions to the projective problem can be obtained from contraction mapping theorem. We finish the construction by solving a reduced finite-dimensional problem. In Section 3, we use a similar method to deal with the Euler equation(the case with $s=1$), and complete the proof of Theorem \ref{thm1} and \ref{thm2}. To conclude the paper, we will investigate the general situation $a\in (0,l/2)$ and prove Theorem \ref{thm3} in Section 4.

\section{Construction for the gSQG equation with $0<s<1$}

In this section we consider the gSQG equation with $0<s<1$, and give proofs for Theorem \ref{thm1} and  \ref{thm2} for this case.

\subsection{Approximate solutions}
To regularize K\'arm\'an point vortex street, we are going to construct a family of solutions $\vartheta_{0,\varepsilon}$ to \eqref{1-6} such that in the sense of measure
\begin{equation*}
	\vartheta_{0,\varepsilon}(x)\rightharpoonup \sum\limits_{k\in \mathbb Z}\boldsymbol\delta_p(x+kl\boldsymbol e_2)-\sum\limits_{k\in \mathbb Z}\boldsymbol\delta_q(x+kl\boldsymbol e_2) \ \ \ \text{as} \ \ \varepsilon\to 0.
\end{equation*}
 We say a function is $l$-periodic, if it takes $l$ as a period in $x_2$ direction. To ensure that the energy of solution is finite, we are to consider the problem in some typical period. For this purpose, we denote $(-\Delta)_*^s$ as $(-\Delta)^s$ acting on $l$-periodic functions and restricted in the typical infinite strip domain  $\mathbb R\times (-l/2,l/2)$ corresponding to the period, which is given by the explicit formula
\begin{equation}\label{2-1}
	(-\Delta)_*^{s}\psi(x)=\int_{\mathbb R\times (-l/2,l/2)}J_s(x-z)\left(\psi(x)-\psi(z)\right)dz,
\end{equation}
where
\begin{equation*}
	J_s(x)=\sum\limits_{k\in \mathbb Z}\frac{C_s}{|x+kl\boldsymbol e_2 |^{2+2s}}, \ \ \ \ \ C_s=\frac{2^{2s}\Gamma(1-s)}{\pi|\Gamma(-s)|},
\end{equation*}
and $\psi(x)$ is some $l$-periodic restricted in $\mathbb R\times (-l/2,l/2)$. We can also denote the inverse of $(-\Delta)_*^s$ as $(-\Delta)_*^{-s}$ with the integral representation
\begin{equation}\label{2-2}
	(-\Delta)_*^{-s}\vartheta(x)=\int_{\mathbb R\times (-l/2,l/2)}K_s(x-z)\vartheta(x)dz, \ \ \ K_s(x)=\sum\limits_{k\in \mathbb Z}G_s(x+kl\boldsymbol e_2),
\end{equation}
for scalar function $\vartheta(x)$ with $supp(\vartheta(x))\subset \mathbb R\times (-l/2,l/2)$.

By the deduction in Section 1, we will consider the following semilinear elliptic problem
\begin{equation}\label{2-3}
	\begin{cases}
	(-\Delta)_*^s\psi=\varepsilon^{(2-2s)\gamma_1-2}(\psi+W_\varepsilon x_1-\varepsilon^{2s-2}\lambda_+)_+^{\gamma_1}\chi_{B_r(p)}\\
	 \,\,\, \ \ \ \ \ \ \ \ \ \ \ \ -\sigma(\varepsilon)^{(2-2s)\gamma_2-2}(-\psi-W_\varepsilon x_1-\sigma(\varepsilon)^{2s-2}\lambda_-)_+^{\gamma_2}\chi_{B_r(q)} \ \ \ \ \ \text{in} \ \ \mathbb R\times (-l/2,l/2),\\
	\psi(x)\to 0 \ \ \ \text{as} \ \  |x_1|\to \infty,
	\end{cases}
\end{equation}
where $\lambda_+$ and $\lambda_-$ are undetermined parameters and will be suitably chosen, $W_\varepsilon$ is the travelling speed of K\'arm\'an vortex street determined by location of $p$, $q$ and $\varepsilon$, $1<\gamma_1,\gamma_2<\frac{2+2s}{2-2s}$ ($\gamma_1,\gamma_2\neq\frac{1}{1-s}$), $\sigma(\varepsilon)$ satisfies assumption \textbf{(H)} with $\tau=\min\{\gamma_2,2\}$, and $r>0$ is a small constant such that $B_r(p)$ and $B_r(q)$ are disjoint. Moreover, we assume
\begin{equation*}
p=(-d, -a), \ \ \  q=(d, a),
\end{equation*}
where $d>0$ for $a=0$; or $d\ge 0$ for $a=l/4$.

Before giving the approximate solutions to \eqref{2-3}, we introduce the following fractional plasma problem, which can be regarded as the limit problem locally.
\begin{equation}\label{2-4}
	\begin{cases}
		(-\Delta)^su=(u-1)^\gamma_+ \ \ \ \ \ \text{in} \ \ \mathbb{R}^2,\\
		u(x)\to 0 \ \ \ \text{as} \ \  |x|\to \infty,
	\end{cases}
\end{equation}
where $0<s<1$ and $1<\gamma<\frac{2+2s}{2-2s}$. In view of \cite{Chan}, \eqref{2-4} has a unique radial solution $U(x)$ known as the ground state with following asymptotic behavior
\begin{equation*}
	\lim\limits_{|x|\to\infty}U(x)=c_sM_\gamma|x|^{-2+2s}, \ \ \ \lim\limits_{|x|\to\infty}U'(|x|)=-\mathcal C_sM_\gamma |x|^{-3+2s},
\end{equation*}
where $M_\gamma=\int_{\mathbb R^2}(U-1)_+^\gamma dx>0$.

Let radial functions $U_1(x)$, $U_2(x)$ be the ground states of \eqref{2-4} with exponent $\gamma=\gamma_1$ and $\gamma=\gamma_2$ respectively. A suitable approximate solution to \eqref{2-3} is
\begin{equation}\label{2-5}
	\Psi_\varepsilon(x)=\varepsilon^{2s-2}\sum\limits_{k\in \mathbb Z}\mu_+^{-\frac{2s}{\gamma_1-1}}U_1\left(\frac{x-p+kl\boldsymbol e_2}{\varepsilon\mu_+}\right)-\sigma(\varepsilon)^{2s-2}\sum\limits_{k\in \mathbb Z}\mu_-^{-\frac{2s}{\gamma_2-1}}U_2\left(\frac{x-q+kl\boldsymbol e_2}{\sigma(\varepsilon)\mu_-}\right),
\end{equation}
where $\mu_+$, $\mu_-$ are positive parameters to be chosen later. To make \eqref{2-5} convergent for every $x\in \mathbb R\times (-l/2,l/2)$, the above sum is understood in the sense
\begin{equation*}
	\begin{split}
		\Psi_\varepsilon&(x)=\varepsilon^{2s-2}\mu_+^{-\frac{2s}{\gamma_1-1}}U_1\left(\frac{x-p}{\varepsilon\mu_+}\right)-\sigma^{2s-2}\mu_-^{-\frac{2s}{\gamma_2-1}}U_2\left(\frac{x-q}{\sigma\mu_-}\right)\\
		&+\lim\limits_{N\to\infty}\sum\limits_{k=1}^N\left(\varepsilon^{2s-2}\sum\limits_{m=\pm k}\mu_+^{-\frac{2s}{\gamma_1-1}}U_1\left(\frac{x-p+ml\boldsymbol e_2}{\varepsilon\mu_+}\right)-\sigma^{2s-2}\sum\limits_{m=\pm k}\mu_-^{-\frac{2s}{\gamma_2-1}}U_2\left(\frac{x-q+ml\boldsymbol e_2}{\sigma\mu_-}\right)\right).
	\end{split}
\end{equation*}
Since as $\varepsilon\to 0$, in the sense of measure,
\begin{equation*}
	(-\Delta)_*^s\Psi_\varepsilon(x)\rightharpoonup M_{\gamma_1}\mu_+^{2-\frac{2s\gamma_1}{\gamma_1-1}}\boldsymbol\delta_p(x)-M_{\gamma_2}\mu_-^{2-\frac{2s\gamma_2}{\gamma_2-1}}\boldsymbol\delta_q(x),
\end{equation*}
we require that $\mu_+$ and $\mu_-$ satisfy
\begin{equation}\label{2-6}
	M_{\gamma_1}\mu_+^{2-\frac{2s\gamma_1}{\gamma_1-1}}=1, \ \ \ M_{\gamma_2}\mu_-^{2-\frac{2s\gamma_2}{\gamma_2-1}}=1,
\end{equation}
which can always be achieved since $\gamma_1,\gamma_2\neq\frac{1}{1-s}$. For simplicity, we will denote
\begin{equation*}
	U_{1,\varepsilon}(x)=\mu_+^{-\frac{2s}{\gamma_1-1}}U_1\left(\frac{x-p}{\varepsilon\mu_+}\right), \ \ \ U_{2,\varepsilon}(x)=\mu_-^{-\frac{2s}{\gamma_2-1}}U_2\left(\frac{x-q}{\sigma(\varepsilon)\mu_-}\right).
\end{equation*}
Then by direct computation, for $x\in B_r(p)$ we have
\begin{equation*}
	\begin{split}
		&(-\Delta)_*^s\Psi_\varepsilon-\varepsilon^{(2-2s)\gamma_1-2}(\Psi_\varepsilon+W_\varepsilon x_1-\varepsilon^{2s-2}\lambda_+)_+^{\gamma_1}\chi_{B_r(p)}\\
		& \ \ \ \ \ \ \ \ \ \ \ \ +\sigma^{(2-2s)\gamma_2-2}(-\Psi_\varepsilon-W_\varepsilon x_1-\sigma^{2s-2}\lambda_-)_+^{\gamma_2}\chi_{B_r(q)}\\
		&=\varepsilon^{-2}\bigg( \left(U_{1,\varepsilon}(x)-\mu_+^{-\frac{2s}{\gamma_1-1}}\right)_+^{\gamma_1}\\
		& \ \ \ \ \ \ \ -\left(\sum\limits_{k\in\mathbb Z}U_{1,\varepsilon}(x+kl\boldsymbol e_2)-\frac{\varepsilon^{2-2s}}{\sigma^{2-2s}}\sum\limits_{k\in\mathbb Z}U_{2,\varepsilon}(x+kl\boldsymbol e_2)+W_\varepsilon\varepsilon^{2-2s}x_1-\lambda_+\right)_+^{\gamma_1}\bigg).
	\end{split}
\end{equation*}
Similarly, for $x\in B_r(q)$ it holds
\begin{equation*}
	\begin{split}
		&(-\Delta)_*^s\Psi_\varepsilon-\varepsilon^{(2-2s)\gamma_1-2}(\Psi_\varepsilon+W_\varepsilon x_1-\varepsilon^{2s-2}\lambda_+)_+^{\gamma_1}\chi_{B_r(p)}\\
		& \ \ \ \ \ \ \ \ \ \ \ \ +\sigma^{(2-2s)\gamma_2-2}(-\Psi_\varepsilon-W_\varepsilon x_1-\sigma^{2s-2}\lambda_-)_+^{\gamma_2}\chi_{B_r(q)}\\
		&=\sigma^{-2}\bigg(-\left(U_{2,\varepsilon}(x)-\mu_-^{-\frac{2s}{\gamma_2-1}}\right)_+^{\gamma_2} \\
		& \ \ \ \ \ \ \ +\left(\sum\limits_{k\in\mathbb Z}U_{2,\varepsilon}(x+kl\boldsymbol e_2)-\frac{\sigma^{2-2s}}{\varepsilon^{2-2s}}\sum\limits_{k\in\mathbb Z}U_{1,\varepsilon}(x+kl\boldsymbol e_2)-W_\varepsilon\sigma^{2-2s}x_1-\lambda_-\right)_+^{\gamma_2}\bigg).
	\end{split}
\end{equation*}
To ensure that $\Psi_\varepsilon(x)$ is a good approximation to the solution to \eqref{2-3}, we choose $\lambda_+$ and $\lambda_-$ in such a way that
\begin{equation}\label{2-7}
	\begin{split}
		\lim\limits_{N\to\infty}\sum\limits_{k=1}^N\bigg(\sum\limits_{m=\pm k}U_{1,\varepsilon}(p+ml\boldsymbol e_2)&-\frac{\varepsilon^{2-2s}}{\sigma^{2-2s}}\sum\limits_{m=\pm k}U_{2,\varepsilon}(p+ml\boldsymbol e_2)\bigg)\\
		&-\frac{\varepsilon^{2-2s}}{\sigma^{2-2s}}U_{2,\varepsilon}(p)-W_\varepsilon\varepsilon^{2-2s}d-\lambda_+=-\mu_+^{-\frac{2s}{\gamma_2-1}}.
	\end{split}
\end{equation}
\begin{equation}\label{2-8}
	\begin{split}
	    \lim\limits_{N\to\infty}\sum\limits_{k=1}^N\bigg(\sum\limits_{m=\pm k}U_{2,\varepsilon}(q+ml\boldsymbol e_2)&-\frac{\sigma^{2-2s}}{\varepsilon^{2-2s}}\sum\limits_{m=\pm k}U_{1,\varepsilon}(q+ml\boldsymbol e_2)\bigg)\\
	    &-\frac{\sigma^{2-2s}}{\varepsilon^{2-2s}}U_{1,\varepsilon}(q)-W_\varepsilon\sigma^{2-2s}d-\lambda_-=-\mu_-^{-\frac{2s}{\gamma_2-1}}.
	\end{split}
\end{equation}
The sums in the above two equalities are convergent, since we have as $|kl|\to \infty$
\begin{equation*}
\sigma^{2-2s}\sum\limits_{m=\pm k}U_{1,\varepsilon}(p+ml\boldsymbol e_2)-\varepsilon^{2-2s}\sum\limits_{m=\pm k}U_{2,\varepsilon}(p+ml\boldsymbol e_2)\thickapprox C|kl|^{2s-4},
\end{equation*}
and
\begin{equation*}
\varepsilon^{2-2s}\sum\limits_{m=\pm k}U_{2,\varepsilon}(q+ml\boldsymbol e_2)-\sigma^{2-2s}\sum\limits_{m=\pm k}U_{1,\varepsilon}(q+ml\boldsymbol e_2)\thickapprox C|kl|^{2s-4}.
\end{equation*}
 Hence $\lambda_+$ and $\lambda_-$ have the following asymptotic behavior
\begin{equation*}
	\lambda_+=\mu_+^{-\frac{2s}{\gamma_1-1}}+O(\varepsilon^{2-2s} ), \ \ \  \lambda_-=\mu_-^{-\frac{2s}{\gamma_2-1}}+O(\sigma^{2-2s}),
\end{equation*}
and the error of the approximation by $\Psi_\varepsilon$ is
\begin{equation}\label{2-9}
	\begin{split}
		&(-\Delta)_*^s\Psi_\varepsilon-\varepsilon^{(2-2s)\gamma_1-2}(\Psi_\varepsilon+W_\varepsilon x_1-\varepsilon^{2s-2}\lambda_+)_+^{\gamma_1}\chi_{B_r(p)}\\
		& \ \ \ \ \ \ \ +\sigma^{(2-2s)\gamma_2-2}(-\Psi_\varepsilon-W_\varepsilon x_1-\sigma^{2s-2}\lambda_-)_+^{\gamma_2}\chi_{B_r(q)}\\
		&=O(\varepsilon^{1-2s})\chi_{B_{L\varepsilon}(p)}+O(\sigma^{1-2s})\chi_{B_{L\sigma}(q)},
	\end{split}
\end{equation}
where $L>0$ is some large constant.

Notice that we are to construct $\vartheta_{0,\varepsilon}(x)$ which is symmetric with respect to $x_2=kl$ when $a=0$, or to $x_2=kl\pm l/4$ when $a=l/4$. For further use, we denote this symmetry restricted in the typical strip $\mathbb R\times (-l/2,l/2)$ as $l$-symmetry. We will focus on $l$-symmetric solutions to \eqref{2-3}, which are small perturbations around $\Psi_\varepsilon(x)$ and can be written as
\begin{equation*}
	\psi_\varepsilon(x)=\Psi_\varepsilon(x)+\omega_\varepsilon(x),\,\,\,x\in \mathbb{R}\times (-l/2,l/2),
\end{equation*}
where $\omega_\varepsilon(x)$ is a family of $l$-symmetric perturbation terms. Actually, by this decomposition we can transform \eqref{2-3} into a equation for $\omega_\varepsilon(x)$, and we will discuss this issue in the rest of this section.

\subsection{The linear theory}
To find suitable $\omega_\varepsilon(x)$ such that $\psi_\varepsilon(x)$ are solutions to \eqref{2-3}, it is necessary to study the linearized operator of \eqref{2-3} at $\Psi_\varepsilon(x)$, which is given by
\begin{equation}\label{3-1}
	\mathbb L_\varepsilon w=(-\Delta)_*^s w-f_u(x,\Psi_\varepsilon) w \ \ \ \ \  \text{in} \ \ \mathbb R\times (-l/2,l/2),
\end{equation}
where and hereafter in this section $f(x,u)$ is the (nonlinear) function in the left hand side of \eqref{2-3}, that is
\[
\begin{split}
f(x,u)&=\varepsilon^{(2-2s)\gamma_1-2}(u+W_\varepsilon x_1-\varepsilon^{2s-2}\lambda_+)_+^{\gamma_1}\chi_{B_r(p)}\\
&\,\,\,\,\,\,-\sigma(\varepsilon)^{(2-2s)\gamma_2-2}(-u-W_\varepsilon x_1-\sigma(\varepsilon)^{2s-2}\lambda_-)_+^{\gamma_2}\chi_{B_r(q)},
\end{split}
\]
so its Fr\'echet derivative at $\Psi_\varepsilon$ is
\begin{equation}\label{3-2}
	\begin{split}
		f_u(x,\Psi_\varepsilon)&=\varepsilon^{(2-2s)\gamma_1-2}\gamma_1(\Psi_\varepsilon+W_\varepsilon x_1-\varepsilon^{2s-2}\lambda_+)_+^{\gamma_1-1}\chi_{B_r(p)}\\
        &\,\,\,\,\,\,+\sigma(\varepsilon)^{(2-2s)\gamma_2-2}\gamma_2(-\Psi_\varepsilon-W_\varepsilon x_1-\sigma(\varepsilon)^{2s-2}\lambda_-)_+^{\gamma_2-1}\chi_{B_r(q)}.
	\end{split}
\end{equation}
Therefore we obtain the following equation for $\omega$ which is equivalent to \eqref{2-3}
\begin{equation}\label{3-3}
	\mathbb L_\varepsilon\omega_\varepsilon=-E_\varepsilon+R_\varepsilon(\omega_\varepsilon) \ \ \ \ \ \ \  \text{in} \ \ \mathbb R\times (-l/2,l/2),
\end{equation}
where
\begin{equation*}
	E_\varepsilon=(-\Delta)_*^s \Psi_\varepsilon-f(x,\Psi_\varepsilon)
\end{equation*}
and
\begin{equation*}
	R_\varepsilon(\omega_\varepsilon)=f(x,\Psi_\varepsilon+\omega_\varepsilon)-f(x,\Psi_\varepsilon)-f_u(x,\Psi_\varepsilon)\omega_\varepsilon.
\end{equation*}

Let $U(x)$ be the ground state solution to the fraction plasma problem \eqref{2-4}, and
\begin{equation*}
	\mathbb L_0 w=(-\Delta)^s w-\gamma(U-1)_+^{\gamma-1}w \ \ \ \ \  \text{in} \ \ \mathbb R^2
\end{equation*}
be the linearization at $U(x)$. To carry out the process of Lyapunov-Schmidlt reduction, we need  the following result of nondegeneracy for the above limiting problem, which is given in \cite{Ao}:
\begin{theorem}\label{thm3-1}
	If $\varphi$ is in the kernel of $\mathbb L_0$ and $\varphi(x)\to 0$ as $|x|\to\infty$, then $\varphi$ is a linear combination of $\frac{\partial U}{\partial x_1}$ and $\frac{\partial U}{\partial x_2}$.
\end{theorem}
Since the parameter $a$ in the $x_2$-coordinate of $p$ and $q$ is $0$ or $l/4$, it is easy to see that $\Psi_\varepsilon$ is $l$-symmetric. From Theorem \ref{thm3-1}, we deduce that the kernel of $\mathbb L_\varepsilon$ is one-dimensional, which is spanned by
\begin{equation*}
	Z_\varepsilon(x)=Z_{1,\varepsilon}(x)-Z_{2,\varepsilon}(x),
\end{equation*}
where
\begin{equation*}
	Z_{1,\varepsilon}(x)=\varepsilon^{2s-2}\sum\limits_{k\in \mathbb Z}\partial_{x_1}U_{1,\varepsilon}(x+kl\boldsymbol e_2), \ \ \  Z_{2,\varepsilon}(x)=\sigma^{2s-2}\sum\limits_{k\in \mathbb Z}\partial_{x_1}U_{2,\varepsilon}(x+kl\boldsymbol e_2).
\end{equation*}
Hence we are to consider the following projected linear problem:
\begin{equation}\label{3-4}
	\begin{cases}
		\mathbb L_\varepsilon\omega_\varepsilon=h(x)+\alpha_\varepsilon f_u(x,\Psi_\varepsilon) Z_\varepsilon(x) \ \ \ \text{in} \ \ \mathbb R\times (-l/2,l/2),\\
		\int_{\mathbb R\times (-l/2,l/2)} f_u(x,\Psi_\varepsilon) Z_\varepsilon(x) \omega_\varepsilon(x)dx=0,\\
		\omega_\varepsilon(x)\to 0 \ \ \ \text{as} \ \ |x_1|\to\infty.
	\end{cases}
\end{equation}
Moreover, we assume that $h(x)$ is $l$-symmetric, and satisfies
\begin{equation}\label{3-5}
supp(h(x))\subset B_{L\varepsilon}(p)\cup B_{L\sigma}(q)
\end{equation}
for some large constant $L>0$. The norms we will use to deal with \eqref{3-4} are
\begin{equation*}
	\|\omega_\varepsilon\|_*=\sup\limits_{x\in \mathbb R\times (-l/2,l/2)} \rho(x)^{-1}|\omega_\varepsilon(x)|,
\end{equation*}
where
\begin{equation*}
	\rho(x)=\left|\frac{1}{\varepsilon^{2-2s}+|x-p|^{2-2s}}-\frac{1}{\sigma^{2-2s}+|x-q|^{2-2s}}\right|+\lim\limits_{N\to\infty}\sum\limits_{k=1}^N\frac{1}{|x+kl\boldsymbol e_2|^{4-2s}},
\end{equation*}
and
\begin{equation*}
	\|h\|_{**}=\sup\limits_{x\in \mathbb R_-\times (-l/2,l/2)}\varepsilon^2 |h(x)|+ \sup\limits_{x\in \mathbb R_+\times (-l/2,l/2)}\sigma^2|h(x)|.
\end{equation*}

We have the following a priori estimate for \eqref{3-4}.
\begin{lemma}\label{lem3-1}
	Assume that $h(x)$ is $l$-symmetric, which satisfies \eqref{3-5} and $\|h\|_{**}<\infty$. Then there exists a small $\varepsilon_0>0$ and a positive constant $C$ such that for any $\varepsilon\in (0,\varepsilon_0)$ and solution pair $(\omega_\varepsilon, \alpha_\varepsilon)$ to \eqref{3-4}, it holds
	\begin{equation}\label{3-6}
		\|\omega_\varepsilon\|_*+(\sigma(\varepsilon))^{-1}|\alpha_\varepsilon|\le C\|h\|_{**}.
	\end{equation}
\end{lemma}
\begin{proof}
	First, let us estimate the second term of the left hand side of \eqref{3-6} and prove
	\begin{equation}\label{3-7}
		(\sigma(\varepsilon))^{-1}|\alpha_\varepsilon|\le C(\|h\|_{**}+o_\varepsilon(1)\|\omega_\varepsilon\|_*).
	\end{equation}
    From \eqref{3-4}, the coefficient $\alpha_\varepsilon$ is given by
    \begin{equation*}
    	\alpha_\varepsilon\int_{\mathbb R\times (-l/2,l/2)}f_u(x,\Psi_\varepsilon)Z_\varepsilon^2dx=\int_{\mathbb R\times (-l/2,l/2)}Z_\varepsilon\mathbb L_\varepsilon\omega_\varepsilon dx-\int_{\mathbb R\times (-l/2,l/2)}hZ_\varepsilon dx.
    \end{equation*}
	According to the expansion of $f_u(x,\Psi_\varepsilon)$ in \eqref{3-2}, we have
	\begin{equation}\label{3-8}
		\begin{split}
		\int_{\mathbb R\times (-l/2,l/2)}f_u(x,\Psi_\varepsilon&)Z_\varepsilon^2dx=(1+o_\varepsilon(1))\varepsilon^{2s-4}\int_{\mathbb R\times (-l/2,l/2)}\gamma_1(U_{1,\varepsilon}-\mu_+^{-\frac{2s}{\gamma_1-1}})_+^{\gamma_1-1}\left(\frac{\partial U_{1,\varepsilon}}{\partial y^1_1}\right)^2dy^1\\
		& \ \ \ \ +(1+o_\varepsilon(1))\sigma^{2s-4}\int_{\mathbb R\times (-l/2,l/2)}\gamma_2(U_{2,\varepsilon}-\mu_-^{-\frac{2s}{\gamma_2-1}})_+^{\gamma_2-1}\left(\frac{\partial U_{2,\varepsilon}}{\partial y^2_1}\right)^2dy^2\\
		& \ \ \ \ =c_1(1+o_\varepsilon(1))\varepsilon^{2s-4}+c_2(1+o_\varepsilon(1))\sigma^{2s-4},
		\end{split}
	\end{equation}
    where $y^1=\frac{x}{\varepsilon\mu_+}$, $y^2=\frac{x}{\sigma\mu_-}$, and $c_1,c_2>0$ are some constants. On the other hand, it holds
    \begin{equation*}
    	\begin{split}
    	&\int_{\mathbb R\times (-l/2,l/2)}Z_\varepsilon\mathbb (-\Delta)_*^s\omega_\varepsilon dx=\int_{\mathbb R\times (-l/2,l/2)}\omega_\varepsilon (-\Delta)_*^sZ_\varepsilon dx\\
    	&=\int_{\mathbb R\times (-l/2,l/2)}\omega_\varepsilon\left( \varepsilon^{-2s}\gamma_1\left(U_{1,\varepsilon}(x)-\mu_+^{-\frac{2s}{\gamma_1-1}}\right)_+^{\gamma_1-1}Z_{1,\varepsilon}-\sigma^{-2s}\gamma_2\left(U_{2,\varepsilon}(x)-\mu_+^{-\frac{2s}{\gamma_2-1}}\right)_+^{\gamma_2-1}Z_{2,\varepsilon}\right)dx.
    	\end{split}
    \end{equation*}
    For $x_1<0$, we have
    \begin{equation*}
    	\begin{split}
    		&\left|\varepsilon^{-2s}\gamma_1\left(U_{1,\varepsilon}(x)-\mu_+^{-\frac{2s}{\gamma_1-1}}\right)_+^{\gamma_1-1}Z_{1,\varepsilon}-\sigma^{-2s}\gamma_2\left(U_{2,\varepsilon}(x)-\mu_+^{-\frac{2s}{\gamma_2-1}}\right)_+^{\gamma_2-1}Z_{2,\varepsilon}-f_u(x,\Psi_\varepsilon)Z_\varepsilon\right|\\
    		&=\left|\varepsilon^{-2s}\gamma_1\left(U_{1,\varepsilon}(x)-\mu_+^{-\frac{2s}{\gamma_1-1}}\right)_+^{\gamma_1-1}Z_{1,\varepsilon}-\varepsilon^{-2s}\gamma_1\left(U_{1,\varepsilon}(x)-\mu_+^{-\frac{2s}{\gamma_1-1}}+O(\varepsilon^{3-2s})\right)_+^{\gamma_1-1}Z_\varepsilon\right|\\
    		&\le C\varepsilon^{-3+(3-2s)\min\{\gamma_1-1,1\}}\chi_{B_{L\varepsilon}(p)}.
    	\end{split}
    \end{equation*}
    Similarly, for $x_1>0$, the term is
    \begin{equation*}
    	\begin{split}
    		&\left|-\sigma^{-2s}\gamma_2\left(U_{2,\varepsilon}(x)-\mu_+^{-\frac{2s}{\gamma_2-1}}\right)_+^{\gamma_2-1}Z_{2,\varepsilon}+\sigma^{-2s}\gamma_2\left(U_{2,\varepsilon}(x)-\mu_+^{-\frac{2s}{\gamma_2-1}}+O(\sigma^{3-2s})\right)_+^{\gamma_2-1}Z_\varepsilon\right|\\
    		&\le C\sigma^{-3+(3-2s)\min\{\gamma_2-1,1\}}\chi_{B_{L\sigma}(q)}.
    	\end{split}
    \end{equation*}
    Hence we derive from H\"older inequality that
    \begin{equation}\label{3-9}
    	\left|\int_{\mathbb R\times (-l/2,l/2)}Z_\varepsilon\mathbb L_\varepsilon\omega_\varepsilon dx\right|\le o_\varepsilon(1)\cdot\|\omega_\varepsilon\|_*\sigma^{2s-3}.
    \end{equation}
    By the definition of norm $\|h\|_{**}$, we also have
    \begin{equation}\label{3-10}
    	\left|\int_{\mathbb R\times (-l/2,l/2)}hZ_\varepsilon dx\right|\le \|h\|_{**}\sigma^{2s-3}.
    \end{equation}
    Then \eqref{3-7} follows directly from \eqref{3-8} \eqref{3-9} and \eqref{3-10}.

    Next, we estimate the first term of the left hand side of \eqref{3-6} and prove
    \begin{equation}\label{3-11}
    	\|\omega_\varepsilon\|_*\le C\|h\|_{**}.
    \end{equation}
    To this end, we will argue by contradiction. Suppose that there exists a sequence $\{\varepsilon_n\}$ satisfying $\varepsilon_n\to 0$, solution pairs $(\omega_{\varepsilon_n}, \alpha_{\varepsilon_n})$ to \eqref{3-4} for some $h_{\varepsilon_n}$, such that
    \begin{equation}\label{3-12}
    		\|\omega_{\varepsilon_n}\|_*=1, \ \ \ \|h_{\varepsilon_n}\|_{**}\to 0 \ \ \ \text{as} \ n\to\infty.
    \end{equation}
    We want to show that for any $L>0$, it always holds
    \begin{equation}\label{3-13}
    	\varepsilon_n^{2-2s}\|\omega_{\varepsilon_n}\|_{L^\infty(B_{L\varepsilon}(p))}+\sigma_n^{2-2s}\|\omega_{\varepsilon_n}\|_{L^\infty(B_{L\sigma_n}(q))}\to 0 \ \ \ \text{as} \ \ n\to\infty,
    \end{equation}
    where $\sigma_n=\sigma(\varepsilon_n)$.

    Suppose \eqref{3-13} is not true, without loss of generality, we assume that the first term satisfies for some constant $\Lambda_0>0$ $$\varepsilon_n^{2-2s}\|\omega_{\varepsilon_n}\|_{L^\infty(B_{L\varepsilon_n}(p))}\ge \Lambda_0.$$
    Set
    \begin{equation*}
    	\tilde \omega_{\varepsilon_n}(y)=\varepsilon_n^{2-2s}\mu_+^{\frac{2s}{\gamma_1-1}}\omega_{\varepsilon_n}(\varepsilon_n\mu_+ y+p).
    \end{equation*}
    From \eqref{3-4}, on every compact set $\tilde\omega_{\varepsilon_n}(y)$ satisfies
    \begin{equation*}
    	\begin{split}
    	(-\Delta)^s \tilde\omega_{\varepsilon_n}(y)&-\gamma_1(U_1-1+O(\varepsilon_n^{3-2s}))_+^{\gamma-1}\tilde\omega_{\varepsilon_n}(y)+o_{\varepsilon_n}(1)\\
    	&=\varepsilon_n^2\mu_+^{\frac{2s\gamma_1}{\gamma_1-1}}h_{\varepsilon_n}(\varepsilon_n\mu_+ y+p)+\varepsilon_n^{-1}\alpha_{\varepsilon_n}\gamma_1\big(U_1-1+O(\varepsilon_n^{3-2s})\big)_+^{\gamma_1-1}\left(\frac{\partial U_1}{\partial y_1}+o_{\varepsilon_n}(1)\right),
    	\end{split}
    \end{equation*}
    which is equivalent to
    \begin{equation*}
    	(- \Delta)^s \tilde\omega_{\varepsilon_n}(y)-\gamma_1(U_1-1)_+^{\gamma_1-1}\tilde\omega_{\varepsilon_n}(y)+o_{\varepsilon_n}(1)=\mathcal R_n(y),
    \end{equation*}
    where
    \begin{equation*}
    \mathcal R_n(y)=\varepsilon_n^2\mu_+^{\frac{2s}{\gamma_1-1}}h_{\varepsilon_n}(\varepsilon_n\mu_+ y+p)+o_{\varepsilon_n}(1)\cdot \tilde\omega_{\varepsilon_n}(y)+\varepsilon_n^{-1}\alpha_{\varepsilon_n}\gamma_1\big(U_1-1+o_{\varepsilon_n}(1)\big)_+^{\gamma_1-1}\left(\frac{\partial U_1}{\partial y_1}+o_{\varepsilon_n}(1)\right).
    \end{equation*}
    Since $\varepsilon_n^2\mu_+^{\frac{2s}{\gamma_1-1}}h_{\varepsilon_n}(y)\to 0$, and $\varepsilon_n^{-1}\alpha_{\varepsilon_n}\le \sigma_n^{-1}\alpha_{\varepsilon_n}\le C(\|h\|_{**}+o_{\varepsilon_n}(1)\|\omega_{\varepsilon_n}\|_*)=o_{\varepsilon_n}(1)$ on every compact set by \eqref{3-7} and \eqref{3-12}, we have $\mathcal R_n(y)\to 0$ as $n\to\infty.$

    Let $n\to \infty$, we may assume that $\tilde\omega_{\varepsilon_n}$ converge uniformly on compact sets to a function $\tilde \omega$ satisfying
    \begin{equation}\label{3-14}
    	\|\tilde \omega\|_{L^\infty(B_{L\mu_+^{-1}}(0))}\ge \Lambda_0\mu_+^{\frac{2s}{\gamma_1-1}}.
    \end{equation}
    However, $\tilde \omega$ is even in $y_2$ direction, and is a solution to
    \begin{equation*}
    	(-\Delta)^s \tilde\omega(y)-\gamma_1(U_1-1)_+^{\gamma_1-1}\tilde\omega(y)=0
    \end{equation*}
    with $\tilde\omega(y)\to 0$ as $|y|\to \infty$. Furthermore, it satisfies the orthogonality condition
    \begin{equation*}
    	\int_{\mathbb{R}^2}\gamma_1(U_1-1)^{\gamma_1-1}\tilde\omega \frac{\partial U_1}{\partial y_1}dy=0.
    \end{equation*}
    According to Theorem \ref{thm3-1}, it must hold $\tilde \omega\equiv 0$, which is a contradiction to \eqref{3-14}. Hence we deduce that $\varepsilon_n^{2-2s}\|\omega_{\varepsilon_n}\|_{L^\infty(B_{L\varepsilon_n}(p))}\to 0$. For the second term in \eqref{3-13}, we can use a similar method to prove $\sigma^{2-2s}\|\omega_{\varepsilon_n}\|_{L^\infty(B_{L\sigma_n}(q))}\to 0$. Hence we have proved \eqref{3-13}.

    In view of \eqref{3-4}, $\omega_{\varepsilon_n}$ satisfies
    \begin{equation*}
    	(- \Delta)_*^s\omega_{\varepsilon_n}=f_u(x,\Psi_0)\omega_{\varepsilon_n}+h_{\varepsilon_n}+\alpha_{\varepsilon_n}f_u(x,\Psi_{\varepsilon_n})Z_{\varepsilon_n}
    \end{equation*}
    Using the explicit formulation \eqref{2-2} of $(-\Delta)_*^{-s}$, we have
    \begin{equation*}
    	\omega_{\varepsilon_n}(x)=\int_{\mathbb R\times (-l/2,l/2)}K_s(x-z)\left(f_u(x,\Psi_0)\omega_{\varepsilon_n}(x)+h_{\varepsilon_n}(x)+\alpha_{\varepsilon_n}f_u(x,\Psi_{\varepsilon_n})Z_{\varepsilon_n}(x)\right)dz
    \end{equation*}
    for the kernel $K_s$ given in \eqref{2-1}, which implies
    \begin{equation*}
    	\rho(x)^{-1}|\omega_{\varepsilon_n}(x)|\le C\left( \varepsilon_n^{2-2s}\|\omega_{\varepsilon_n}\|_{L^\infty(B_{L\varepsilon_n}(p))}+
    \sigma_n^{2-2s}\|\omega_{\varepsilon_n}\|_{L^\infty(B_{L\sigma_n}(q))}+\|h_{\varepsilon_n}\|_{**}+\sigma_n^{-1}\alpha_{\varepsilon_n}\right).
    \end{equation*}
    Hence if we combine \eqref{3-7} \eqref{3-12} \eqref{3-13}, we can obtain $\|\omega_{\varepsilon_n}\|_*\to 0$ as $n\to\infty$, which is a contradiction to \eqref{3-12} and yields \eqref{3-11}. To finish our proof, we notice that \eqref{3-6} is the consequence of \eqref{3-7} and \eqref{3-11}.
\end{proof}

Using the a priori estimate given in Lemma \ref{lem3-1}, we have the following result for \eqref{3-4}.
\begin{lemma}\label{lem3-2}
	Assume that $h(x)$ is $l$-symmetric, which satisfies \eqref{3-5} and $\|h\|_{**}<\infty$. Then there exists a small $\varepsilon_0>0$ such that for any $\varepsilon\in (0,\varepsilon_0)$, \eqref{3-4} has a unique solution $\omega_\varepsilon=T_\varepsilon h$, where $T_\varepsilon$ is a linear operator of $h$. Moreover, there exists a constant $C>0$ independent of $\varepsilon$ such that
	\begin{equation*}
		\|\omega_\varepsilon\|_*\le C\|h\|_{**}.
	\end{equation*}
\end{lemma}
\begin{proof}
	Denote the Hilbert space
	\begin{equation*}
		H:=\left\{ g\in \dot H^s(\mathbb R\times (-l/2,l/2)) \ : \   g \ \text{is} \ l\text{-symmetric}, \ \int_{\mathbb R\times (-l/2,l/2)} f_u(x,\Psi_\varepsilon) Z_\varepsilon(x) g(x) dx=0\right\}
	\end{equation*}
    endowed with the inner product
    \begin{equation*}
    [u,g]=\int_{\mathbb R\times (-l/2,l/2)}\int_{\mathbb R\times (-l/2,l/2)}J_s(x-z)\left(u(x)-u(z)\right)\left(g(x)-g(z)\right)dxdz,
    \end{equation*}
	where the kernel  $J_s(x)$ is defined in \eqref{2-1}. Then we can express \eqref{3-4} in a weak form, namely, to find $\omega_\varepsilon\in H$ such that
	\begin{equation*}
		[\omega_\varepsilon,g]=\langle f_u(x,\Psi_\varepsilon)\omega_\varepsilon+h , g  \rangle, \ \ \ \ \ \forall \, g\in H.
	\end{equation*}
    According to Riesz's representation theorem, the above equation has a equivalent operational form
    \begin{equation*}
    	\omega_\varepsilon=(-\Delta)_*^{-s}(f_u(x,\Psi_\varepsilon)\omega_\varepsilon)+(-\Delta_*)^{-s}h.
    \end{equation*}
	Notice that  $(-\Delta)_*^{-s}(f_u(x,\Psi_\varepsilon)(\cdot))$ is a compact operator on $H$. By Fredholm's alternative, this equation has a unique solution for any $h$ if the homogeneous equation
	\begin{equation*}
		\omega_\varepsilon=(-\Delta)_*^{-s}(f_u(x,\Psi_\varepsilon)\omega_\varepsilon)
	\end{equation*}
	has only trivial solution in $H$, which can be obtained by Lemma \ref{lem3-1}. The estimate $\|\omega_\varepsilon\|_*\le C\|h\|_{**}$ follows from \eqref{3-11}. Hence the proof is complete.	
\end{proof}

\subsection{The reduction}
 To solve \eqref{3-3}, we will first solve \eqref{3-4} for
 \begin{equation}\label{4-1}
 	h(x)=-E_\varepsilon+R_\varepsilon(\omega_\varepsilon).
 \end{equation}
Then we will deal with a one-dimensional problem so that $\alpha_\varepsilon=0$, which can be achieved by choosing suitable travelling speed $W_\varepsilon$. This process is known as the Lyapunov-Schmidlt reduction. For the solvability of \eqref{3-4} and \eqref{4-1}, we have the following lemma.
\begin{lemma}\label{lem4-1}
	There are $\varepsilon_0>0$ and $r_0>0$ such that for $\varepsilon\in(0,\varepsilon_0)$ there exists a unique solution $\omega_\varepsilon$ to \eqref{3-4} and \eqref{4-1} in the ball $\|\omega_\varepsilon\|_*\le r_0$. Moreover, it holds
	\begin{equation}\label{4-2}
		\|\omega_\varepsilon\|_*\le C\varepsilon^{3-2s}
	\end{equation}
    for some constant $C>0$, and $\omega_\varepsilon$ is continuous with respect to $\varepsilon$.
\end{lemma}
\begin{proof}
	Since $h_\varepsilon(x)$ is given in \eqref{4-1}, it is easy to verify that $h(x)$ is $l$-symmetric and satisfies \eqref{3-5}. Hence from Lemma \ref{lem3-2}, for the $h(x)$, we have the estimate
	\begin{equation}\label{4-3}
		\|T_\varepsilon h\|_*\le C\|h\|_{**}.	
	\end{equation}
	
	Denote
	\begin{equation*}
		X:=\{u\in L^\infty(\mathbb\times(-l/2,l/2)) \ : \ u \ \text{is} \ l\text{-symmetric}, \ \|u\|_*<\infty, \}
	\end{equation*}
	endowed with the norm $\|\cdot\|_*$, and $\mathcal A_\varepsilon: X\to X$ the operator given by
	\begin{equation*}
		\mathcal A_\varepsilon\omega_\varepsilon:=T_\varepsilon(-E_\varepsilon+R_\varepsilon(\omega_\varepsilon)).
	\end{equation*}
	If we let
	\begin{equation*}
		\mathcal B_{r_0}:=\{\omega_\varepsilon\in X \ : \ \|\omega_\varepsilon\|_*\le r_0\}
	\end{equation*}
	be a closed neighborhood of the origin in $X$, then solve equation \eqref{3-4} is equivalent to find a fixed point of $\mathcal A_\varepsilon$
in $\mathcal B_{r_0}$,
	\begin{equation*}
		\mathcal A_\varepsilon\omega_\varepsilon=\omega_\varepsilon.
	\end{equation*}

	In the following we prove that $\mathcal A_\varepsilon$ does have a fixed point in $\mathcal B_{r_0}$ by showing that $\mathcal A_\varepsilon$ is a contraction
map in $\mathcal B_{r_0}$. First we show that $\mathcal A_\varepsilon$ maps $\mathcal B_{r_0}$ into itself.
	From \eqref{2-9}, we have
	\begin{equation}\label{4-4}
		\|E_\varepsilon\|_{**}\le C\varepsilon^{3-2s}.
	\end{equation}
	Then we are to estimate $\|R_\varepsilon(\omega_\varepsilon)\|_{**}$. To this aim, we can split $R_\varepsilon(\omega_\varepsilon)$ into two terms
	\begin{equation*}
		R_\varepsilon(\omega_\varepsilon)=R_{1,\varepsilon}(\omega_\varepsilon)+R_{2,\varepsilon}(\omega_\varepsilon),
	\end{equation*}
	where
	\begin{equation*}
		\begin{split}
			R_{1,\varepsilon}(\omega_\varepsilon)=\varepsilon^{(2-2s)\gamma_1-2}&\bigg((\Psi_\varepsilon+\omega_\varepsilon+W_\varepsilon x_1-\varepsilon^{2s-2}\lambda_+)_+^{\gamma_1}\\
			&-(\Psi_\varepsilon+W_\varepsilon x_1-\varepsilon^{2s-2}\lambda_+)_+^{\gamma_1}-\gamma_1(\Psi_\varepsilon+W_\varepsilon x_1-\varepsilon^{2s-2}\lambda_+)_+^{\gamma_1-1}\omega_\varepsilon\bigg)\chi_{B_{L\varepsilon}(p)}
		\end{split}
	\end{equation*}
    and
    \begin{equation*}
    	\begin{split}
    		R_{2,\varepsilon}(\omega_\varepsilon)=\sigma^{(2-2s)\gamma_2-2}&\bigg((-\Psi_\varepsilon-\omega_\varepsilon-W_\varepsilon x_1-\sigma^{2s-2}\lambda_+)_+^{\gamma_2}\\
    		&-(-\Psi_\varepsilon-W_\varepsilon x_1-\sigma^{2s-2}\lambda_+)_+^{\gamma_2}+\gamma_2(-\Psi_\varepsilon-W_\varepsilon x_1-
            \sigma^{2s-2}\lambda_+)_+^{\gamma_2-1}\omega_\varepsilon\bigg)\chi_{B_{L\sigma}(p)}.
    	\end{split}
    \end{equation*}
	By the choice of $\lambda_+$ in \eqref{2-7}, we have
	\begin{equation*}
		\begin{split}
			R_{1,\varepsilon}(\omega_\varepsilon)&=\varepsilon^{-2}\bigg((\sum\limits_{k\in \mathbb Z}U_{1,\varepsilon}(x+kl\boldsymbol e_2)-\frac{\varepsilon^{2-2s}}{\sigma^{2-2s}}\sum\limits_{k\in \mathbb Z}U_{1,\varepsilon}(x+kl\boldsymbol e_2)+\varepsilon^{2-2s}\omega_\varepsilon-\mu_+^{-\frac{2s}{\gamma_1-1}}+O(\varepsilon^{3-2s}))_+^{\gamma_1}\\
			& \ \ \ -(\sum\limits_{k\in \mathbb Z}U_{1,\varepsilon}(x+kl\boldsymbol e_2)-\frac{\varepsilon^{2-2s}}{\sigma^{2-2s}}\sum\limits_{k\in \mathbb Z}U_{1,\varepsilon}(x+kl\boldsymbol e_2)-\mu_+^{-\frac{2s}{\gamma_1-1}}+O(\varepsilon^{3-2s}))_+^{\gamma_1}\\
			&-\gamma_1(\sum\limits_{k\in \mathbb Z}U_{1,\varepsilon}(x+kl\boldsymbol e_2)-\frac{\varepsilon^{2-2s}}{\sigma^{2-2s}}\sum\limits_{k\in \mathbb Z}U_{1,\varepsilon}(x+kl\boldsymbol e_2)-\mu_+^{-\frac{2s}{\gamma_1-1}}+O(\varepsilon^{3-2s}))_+^{\gamma_1-1}\varepsilon^{2-2s}\omega_\varepsilon\bigg)\chi_{B_{L\varepsilon}(p)},
		\end{split}
	\end{equation*}
    which yields $\|R_{1,\varepsilon}(\omega_\varepsilon)\|_{**}\le C\|\omega_\varepsilon\|_*^{\min\{\gamma_1,2\}}$. Using a similar method, we can also prove $\|R_{2,\varepsilon}(\omega_\varepsilon)\|_{**}\le C\|\omega_\varepsilon\|_*^{\min\{\gamma_2,2\}}$. So we conclude that
    \begin{equation}\label{4-5}
    	\|R_\varepsilon(\omega_\varepsilon)\|_{**}\le C\|\omega_\varepsilon\|_*^{\min\{\gamma_1,\gamma_2,2\}}.
    \end{equation}
    If we combine \eqref{4-3} \eqref{4-4} and \eqref{4-5}, we deduce that for $\omega_\varepsilon\in \mathcal B_{r_0}$, it holds
    \begin{equation*}
    	\|\mathcal A_\varepsilon\omega_\varepsilon\|\le C\varepsilon^{3-2s}+Cr_0^{\min\{\gamma_1,\gamma_2,2\}}.
    \end{equation*}
    As a result, $\mathcal A_\varepsilon$ maps $\mathcal B_{r_0}$ into itself if we choose $\varepsilon$ and $r_0$ sufficiently small.
	
	On the other hand, for $\omega_\varepsilon^1,\omega_\varepsilon^2\in \mathcal B_{r_0}$ we have
	\begin{equation*}
		\|R_\varepsilon(\omega_\varepsilon^1)-R_\varepsilon(\omega_\varepsilon^2)\|_{**}\le C(\|\omega_\varepsilon^1\|_*^{\min\{\gamma_1-1,\gamma_2-1,1\}}+\|\omega_\varepsilon^2\|_*^{\min\{\gamma_1-1,\gamma_2-1,1\}})\|\omega_\varepsilon^1-\omega_\varepsilon^2\|_*.
	\end{equation*}
    Hence it holds
    \begin{equation*}
    	\|\mathcal A_\varepsilon\omega_\varepsilon^1-\mathcal A_\varepsilon\omega_\varepsilon^2\|_{**}\le Cr_0^{\min\{\gamma_1-1,\gamma_2-1,1\}}\|\omega_\varepsilon^1-\omega_\varepsilon^2\|_*,
    \end{equation*}
	and $\mathcal A_\varepsilon$ is a contraction mapping from $\mathcal B_{r_0}$ into itself if $r_0$ is sufficiently small. Thus \eqref{3-4} and \eqref{4-1} admits a unique solution $\omega_\varepsilon\in \mathcal B_{r_0}$.
	
	According to the estimate \eqref{4-4}, we derive
	\begin{equation}\label{4-6}
		\|\omega_\varepsilon\|_*\le C\|E_\varepsilon\|_{**}\le C\varepsilon^{3-2s}.
	\end{equation}
    Since $E_\varepsilon$ and $R_\varepsilon(\omega_\varepsilon)$ depend continuously on $\varepsilon$, we see that $\omega_\varepsilon$ is continuous with respect to $\varepsilon$ by the fixed point characterization. So we have finished the proof.
\end{proof}

We have already obtained a solution $\psi_\varepsilon(x)=\Psi_\varepsilon(x)+\omega_\varepsilon(x)$ to
\begin{equation}\label{4-7}
	\begin{cases}
		(-\Delta)_*^{s}\psi_\varepsilon=f(x,\psi_\varepsilon)+\alpha_\varepsilon f_u(x,\Psi_\varepsilon) Z_\varepsilon(x) \ \ \ \text{in} \ \ \mathbb R\times (-l/2,l/2),\\
		\psi_\varepsilon(x)\to 0 \ \ \ \text{as} \ \ |x_1|\to\infty.
	\end{cases}
\end{equation}
If we multiply the first equation of \eqref{4-7} by $Z_\varepsilon$ and integrate over $\mathbb R\times (-l/2,l/2)$, we deduce that
\begin{equation*}
	\alpha_\varepsilon\int_{\mathbb R\times (-l/2,l/2)}f_u(x,\Psi_\varepsilon)Z_\varepsilon^2dx=\int_{\mathbb R\times (-l/2,l/2)}\left((-\Delta)_*^{s}\psi_\varepsilon-f(x,\psi_\varepsilon)\right)Z_\varepsilon dx.
\end{equation*}
To make the right hand side of above equality being zero, we will use the following lemma later.
\begin{lemma}\label{lem4-2}
	It holds
	\begin{equation*}
		\int_{\mathbb R\times (-l/2,l/2)}\left((-\Delta)_*^s\psi_\varepsilon-f(x,\psi_\varepsilon)\right)Z_\varepsilon dx=C\left(\mathcal C_s\lim\limits_{N\to\infty}\sum\limits_{|k|\le N}\frac{p_1-q_1}{|q-p+kl\boldsymbol e_2|^{4-2s}}+W_\varepsilon\right)+o_\varepsilon(1),
	\end{equation*}
    where $C>0$ is a constant independent of $\varepsilon$, and $p_1$, $q_1$ denote the first coordinates of $p$, $q$ respectively.
\end{lemma}
\begin{proof}
	From \eqref{3-4} and \eqref{4-1}, we have
	\begin{equation*}
		(-\Delta)_*^{s}\psi_\varepsilon-f(x,\psi_\varepsilon)=\mathbb{L}_\varepsilon\psi_\varepsilon+E_\varepsilon-R_\varepsilon(\omega_\varepsilon)  \ \ \ \ \ \ \ \text{in} \ \ \mathbb R\times (-l/2,l/2).
	\end{equation*}
    Multiplying this equality by $Z_\varepsilon$ and integrating over $\mathbb R\times (-l/2,l/2)$, one can derive
    \begin{equation*}
    	\begin{split}
    	\int_{\mathbb R\times (-l/2,l/2)}\big((-\Delta)_*^{s}\psi_\varepsilon&-f(x,\psi_\varepsilon)\big)Z_\varepsilon dx\\
    	&=\int_{\mathbb R\times (-l/2,l/2)}\mathbb{L}_\varepsilon\omega_\varepsilon Z_\varepsilon dx+\int_{\mathbb R\times (-l/2,l/2)}E_\varepsilon Z_\varepsilon dx-\int_{\mathbb R\times (-l/2,l/2)}R_\varepsilon(\omega_\varepsilon)Z_\varepsilon dx.
    	\end{split}
    \end{equation*}
	
	We first deal with the term $\int_{\mathbb R\times (-l/2,l/2)}E_\varepsilon Z_\varepsilon dx$. By our choice of $\lambda_+$ and $\lambda_-$ in \eqref{2-7} and \eqref{2-8}, we can split $E_\varepsilon$ as
	\begin{equation*}
		E_\varepsilon=E_{1,\varepsilon}+E_{2,\varepsilon},
	\end{equation*}
    where
    \begin{equation*}
    	\begin{split}
    		E_{1,\varepsilon}&=\varepsilon^{-2}\chi_{B_{L\varepsilon}(p)}\bigg((U_{1,\varepsilon}(x)-\mu_+^{-\frac{2s}{\gamma_1-1}})^{\gamma_1}_+-(U_{1,\varepsilon}(x)-\mu_+^{-\frac{2s}{\gamma_1-1}}\\
    		&+\sum\limits_{k\neq0}U_{1,\varepsilon}(x+kl\boldsymbol e_2)-\frac{\varepsilon^{2-2s}}{\sigma^{2-2s}}\sum\limits_{k\in\mathbb Z}U_{2,\varepsilon}(x+kl\boldsymbol e_2)-\sum\limits_{k\neq0}U_{1,\varepsilon}(p+kl\boldsymbol e_2)\\
    		&+\frac{\varepsilon^{2-2s}}{\sigma^{2-2s}}\sum\limits_{k\in\mathbb Z}U_{2,\varepsilon}(p+kl\boldsymbol e_2)
    		+W_\varepsilon\varepsilon^{2-2s}(x_1-d))^{\gamma_1}_+\bigg),
    	\end{split}
    \end{equation*}
	and
	\begin{equation*}
		\begin{split}
			E_{2,\varepsilon}&=\sigma^{-2}\chi_{B_{L\sigma}(q)}\bigg(-(U_{2,\varepsilon}(x)-\mu_-^{-\frac{2s}{\gamma_2-1}})^{\gamma_2}_++(U_{2,\varepsilon}(x)-\mu_-^{-\frac{2s}{\gamma_2-1}}\\
			&+\sum\limits_{k\neq0}U_{2,\varepsilon}(x+kl\boldsymbol e_2)-\frac{\sigma^{2-2s}}{\varepsilon^{2-2s}}\sum\limits_{k\in\mathbb Z}U_{1,\varepsilon}(x+kl\boldsymbol e_2)-\sum\limits_{k\neq0}U_{2,\varepsilon}(q+kl\boldsymbol e_2)\\
			&+\frac{\varepsilon^{2-2s}}{\sigma^{2-2s}}\sum\limits_{k\in\mathbb Z}U_{1,\varepsilon}(q+kl\boldsymbol e_2)-W_\varepsilon\varepsilon^{2-2s}(x_1-d))^{\gamma_2}_+\bigg).
		\end{split}
     \end{equation*}	
	 Since
	 \begin{equation*}
	 	|E_{1,\varepsilon}|\le C\varepsilon^{-2}\cdot\varepsilon^{3-2s} \ \ \  \text{and} \ \ \ |E_{2,\varepsilon}|\le C\sigma^{-2}\cdot\sigma^{3-2s},
	 \end{equation*}
	 we can show that
	 \begin{equation*}
	 	\int_{\mathbb R\times (-l/2,l/2)}E_{1,\varepsilon} Z_{2,\varepsilon} dx\le C\varepsilon^{3-2s}, \ \ \  \int_{\mathbb R\times (-l/2,l/2)}E_{2,\varepsilon} Z_{1,\varepsilon} dx\le C\sigma^{3-2s}.
	 \end{equation*}
     From Taylor's formula, it holds
     \begin{equation*}
     	\begin{split}
     		E_{1,\varepsilon}=&-\varepsilon^{-2}\chi_{B_{L\varepsilon}(p)}(U_{1,\varepsilon}(x)-\mu_+^{-\frac{2s}{\gamma_1-1}})^{\gamma_1-1}_+\bigg(\sum\limits_{k\neq0}U_{1,\varepsilon}(x+kl\boldsymbol e_2)\\
     		&-\frac{\varepsilon^{2-2s}}{\sigma^{2-2s}}\sum\limits_{k\in\mathbb Z}U_{2,\varepsilon}(x+kl\boldsymbol e_2)-\sum\limits_{k\neq0}U_{1,\varepsilon}(p+kl\mathbf e_2)+\frac{\varepsilon^{2-2s}}{\sigma^{2-2s}}\sum\limits_{k\in\mathbb Z}U_{2,\varepsilon}(p+kl\boldsymbol e_2)\\
     		&+W_\varepsilon\varepsilon^{2-2s}(x_1-d)\bigg) +\varepsilon^{-2}\cdot\varepsilon^{(3-2s)(\gamma_1-1)}\chi_{B_{L\varepsilon}(p)}.
     	\end{split}
     \end{equation*}
	To compute $\int_{\mathbb R\times (-l/2,l/2)}E_{1,\varepsilon} Z_{1,\varepsilon} dx$, we can integrate by parts and use the asymptotic behavior of $U_{1,\varepsilon}, U_{2,\varepsilon}$ to obtain
	\begin{equation*}
		\begin{split}
			\int_{\mathbb R\times (-l/2,l/2)}&E_{1,\varepsilon} Z_{1,\varepsilon} dx=\varepsilon^{-2}\int_{\mathbb R\times (-l/2,l/2)}(U_{1,\varepsilon}(x)-\mu_+^{-\frac{2s}{\gamma_1-1}})^{\gamma_1}_+\\
			& \ \ \ \ \ \ \ \ \ \ \left(\varepsilon^{2s-2}\sum\limits_{k\neq0}\partial_{x_1}U_{1,\varepsilon}(x+kl\boldsymbol e_2)
			-\sigma^{2s-2}\sum\limits_{k\in\mathbb Z}\partial_{x_1}U_{2,\varepsilon}(x+kl\boldsymbol e_2)+W_\varepsilon+O(\varepsilon)\right)dx\\
			&=\varepsilon^{-2}\int_{\mathbb R\times (-l/2,l/2)}(U_{1,\varepsilon}(x)-\mu_+^{-\frac{2s}{\gamma_1-1}})^{\gamma_1}_+\\
			& \ \ \ \ \ \ \ \ \ \ \left((2-2s)c_s\lim\limits_{N\to\infty}\sum\limits_{|k|\le N}\frac{p_1-q_1}{|p-q+kl\boldsymbol e_2|^{4-2s}}+W_\varepsilon+O(\varepsilon)\right)dx\\
			&=C_1\left(\mathcal C_s\lim\limits_{N\to\infty}\sum\limits_{|k|\le N}\frac{p_1-q_1}{|p-q+kl\boldsymbol e_2|^{4-2s}}+W_\varepsilon\right)+O(\varepsilon),		
		\end{split}
	\end{equation*}
	where we have used symmetry to derive $\lim\limits_{N\to\infty}\sum\limits_{|k|\le N}\partial_{x_1}U_{1,\varepsilon}(p+kl\boldsymbol e_2)=0$, $C_1>0$ is a constant, and $p_1,q_1$ denote the first coordinates of $p,q$ respectively. Similarly, for some $C_2>0$ we can get
	\begin{equation*}
	\begin{split}
		-\int_{\mathbb R\times (-l/2,l/2)}E_{2,\varepsilon} Z_{2,\varepsilon} dx=C_2\left(\mathcal C_s\lim\limits_{N\to\infty}\sum\limits_{|k|\le N}\frac{-q_1+p_1}{|q-p+kl\boldsymbol e_2|^{4-2s}}+W_\varepsilon\right)+O(\sigma).		
	\end{split}
    \end{equation*}
    Hence we have
    \begin{equation}\label{4-8}
    	\begin{split}
    		\int_{\mathbb R\times (-l/2,l/2)}E_\varepsilon Z_\varepsilon dx=C\left(\mathcal C_s\lim\limits_{N\to\infty}\sum\limits_{|k|\le N}\frac{p_1-q_1}{|q-p+kl\boldsymbol e_2|^{4-2s}}+W_\varepsilon\right)+o_\varepsilon(1).
    	\end{split}
    \end{equation}
	
	Recall that $\sigma(\varepsilon)$ satisfies assumption \textbf{(H)} with $\tau=\min\{\gamma_2,2\}$. Since $\|\omega_\varepsilon\|_*\le C\varepsilon^{3-2s}$, for the term $\int_{\mathbb R\times (-l/2,l/2)}R_\varepsilon(\omega_\varepsilon)Z_\varepsilon dx$ we have
	\begin{equation}\label{4-9}
		\begin{split}
		\int_{\mathbb R\times (-l/2,l/2)}R_\varepsilon(\omega_\varepsilon)Z_\varepsilon dx&=\int_{\mathbb R_-\times (-l/2,l/2)}R_\varepsilon(\omega_\varepsilon)Z_\varepsilon dx+\int_{\mathbb R_+\times (-l/2,l/2)}R_\varepsilon(\omega_\varepsilon)Z_\varepsilon dx\\
		&\le C\|\omega_\varepsilon\|_*^{\min\{\gamma_1,2\}}\varepsilon^{-3+2s}+C\|\omega_\varepsilon\|_*^{\min\{\gamma_2,2\}}\sigma^{-3+2s}\\
		&\le C\varepsilon^{(3-2s)\min\{\gamma_1-1, 1\}}+C\varepsilon^{(3-2s)\min\{\gamma_2, 2\}}\sigma^{-3+2s}=o_\varepsilon(1).
	    \end{split}
	\end{equation}
	To deal with the last term $\int_{\mathbb R\times (-l/2,l/2)}\mathbb{L}_\varepsilon\omega_\varepsilon Z_\varepsilon dx$, we can use the estimate in the proof of Lemma \ref{lem3-1} to deduce that
	\begin{equation}\label{4-10}
		\begin{split}
		\int_{\mathbb R\times (-l/2,l/2)}\mathbb{L}_\varepsilon\omega_\varepsilon Z_\varepsilon dx&=\int_{\mathbb R_-\times (-l/2,l/2)}\mathbb{L}_\varepsilon\omega_\varepsilon Z_\varepsilon dx+\int_{\mathbb R_+\times (-l/2,l/2)}\mathbb{L}_\varepsilon\omega_\varepsilon Z_\varepsilon dx\\
		&\le C\|\omega_\varepsilon\|_*\varepsilon^{(3-2s)\min\{\gamma_1-2,0\}}+C\|\omega_\varepsilon\|_*\sigma^{(3-2s)\min\{\gamma_2-2,0\}}\\
		&\le C\varepsilon^{(3-2s)\min\{\gamma_1-1,1\}}+C\varepsilon^{3-2s}\sigma^{(3-2s)\min\{\gamma_2-2,0\}}=o_\varepsilon(1).
		\end{split}
	\end{equation}
	Finally, if we combine \eqref{4-8} \eqref{4-9} and \eqref{4-10}, then the proof is complete.	
\end{proof}

Now we are ready to given proofs for Theorem \ref{thm1} and  \ref{thm2} with $0<s<1$.\\
\vspace{0.3cm}

{\bf Proof of Theorem \ref{thm1} and  \ref{thm2} with $0<s<1$:}
In view of \eqref{4-7}, to obtain a family of desired solutions to \eqref{2-3}, we only need to find suitable $W_\varepsilon$ so that the corresponding $\alpha_\varepsilon=0$. Notice that we already have
\begin{equation*}
\int_{\mathbb R\times (-l/2,l/2)}f_u(x,\Psi_\varepsilon)Z_\varepsilon^2dx>0.
\end{equation*}
By Lemma \ref{lem4-2}, $\alpha_\varepsilon=0$ is equivalent to the variational characterization
\begin{equation*}
	W_\varepsilon=\mathcal C_s\lim\limits_{N\to\infty}\sum\limits_{|k|\le N}\frac{-p_1+q_1}{|q-p+kl\boldsymbol e_2|^{4-2s}}+o_\varepsilon(1).
\end{equation*}
Recall the definitions of $p$ and $q$. When $a=0$, the condition for $W_\varepsilon$ is
\begin{equation*}
	W_\varepsilon=\mathcal C_s\lim\limits_{N\to\infty}\sum\limits_{|k|\le N}\frac{2d}{(4d^2+k^2l^2)^{2-s}}+o_\varepsilon(1);
\end{equation*}
while for $a=l/4$, it must hold
\begin{equation*}
	W_\varepsilon=\mathcal C_s\lim\limits_{N\to\infty}\sum\limits_{|k|\le N}\frac{2d}{(4d^2+(kl+\frac{l}{2})^2)^{2-s}}+o_\varepsilon(1).
\end{equation*}
By \eqref{2-3} and the periodic setting, the weak convergence of solutions is obvious. By the choice of $p,q$, the existence of $W_\varepsilon$ follows and the corresponding $\alpha_\varepsilon=0$. The $C^1$ property of $\vartheta_{0,\varepsilon}$ can be deduced from the standard  regularity theory for elliptic equations. Hence we have completed the proof.
\qed

\begin{remark}
	In \cite{Ao}, the condition $\alpha_\varepsilon=0$ is described as $\psi_\varepsilon$ being a critical point of some energy functional $\mathcal E(\psi)$, which degenerates in $x_2$ direction.  Since the vortex pair constructed in \cite{Ao} has an odd symmetry, this description is appropriate and vivid. In our situation, vortices on different sides of the street have an different energy blow up rate if $\delta(\varepsilon)=o_\varepsilon(1)$. However, our description for $\alpha_\varepsilon=0$ does make sense, because assumption \textbf{(H)} ensures that the small terms caused by $R_\varepsilon$ or $\mathbb L_\varepsilon\omega_\varepsilon$ are of order $o_\varepsilon(1)$, and can not exceed the secondary term in energy functional, which is of order $O_\varepsilon(1)$ and determines $W_\varepsilon$.
\end{remark}

\section{Construction for the Euler equation}

In this section we consider the remaining case $s = 1$ for gSQG equation, namely the Euler equation and give proofs for Theorem \ref{thm1} and  \ref{thm2} in this case.

\subsection{Approximate solutions}
 As we have done in Section 2, we are going to obtain a series of $l$-symmetric solutions to the following semilinear elliptic problem
\begin{equation}\label{5-1}
	\begin{cases}
	-\Delta\psi=\varepsilon^{-2}\left(\psi+W_\varepsilon x_1-\frac{\lambda_+}{2\pi}\ln\frac{1}{\varepsilon}\right)_+^{\gamma_1}\chi_{B_r(p)}\\
	 \ \ \ \ \ \ \ \ \ \ \ -\sigma(\varepsilon)^{-2}\left(-\psi-W_\varepsilon x_1-\frac{\lambda_-}{2\pi}\ln\frac{1}{\sigma(\varepsilon)}\right)_+^{\gamma_2}\chi_{B_r(q)} \ \ \ \ \ \text{in} \ \ \mathbb R\times (-l/2,l/2),\\
	\psi(x)\to 0 \ \ \ \text{as} \ \  |x_1|\to \infty,
    \end{cases}
\end{equation}
where $\lambda_+$ and $\lambda_-$ are undetermined parameters and will be suitably chosen, $W_\varepsilon$ is the travelling speed of K\'arm\'an vortex street determined by location of $p$, $q$ and $\varepsilon$, $1<\gamma_1,\gamma_2<\infty$, $\sigma(\varepsilon)$ satisfies assumption \textbf{(H)} with $\tau=\min\{\gamma_2,2\}$, and $r>0$ is a small constant such that $B_r(p)$ and $B_r(q)$ are disjoint. We still assume
\begin{equation*}
	p=(-d, -a), \ \ \  q=(d, a),
\end{equation*}
where $d>0$ for $a=0$; or $d\ge 0$ for $a=l/4$.

Suppose $V(x)=V(|x|)$ is the unique radial solution of
\begin{equation}\label{5-2}
	-\Delta V=V^\gamma, \ \ V\in H^1_0(B_1(0)), \ \ V>0 \ \ \text{in} \ B_1(0).
\end{equation}
For $s_+,s_-$ undetermined, let
\begin{equation*}
	V_{1,\varepsilon}(x)=\left\{
	\begin{array}{lll}
		\frac{1}{2\pi}\ln\frac{1}{\varepsilon}+\varepsilon^{\frac{2}{\gamma_1-1}}s_+^{-\frac{2}{\gamma_1-1}}V_1(\frac{|x-p|}{s_+}) \ \ \ \ \ & |x-p|\le s_+,\\
		\frac{1}{2\pi}\ln\frac{1}{\varepsilon}\cdot\frac{\ln |x-p|}{\ln s_+} & |x-p|>s_+,
	\end{array}
	\right.
\end{equation*}
where $V_1$ is the solution to \eqref{5-2} with exponent $\gamma=\gamma_1$, and \begin{equation*}
	V_{2,\varepsilon}(x)=\left\{
	\begin{array}{lll}
		\frac{1}{2\pi}\ln\frac{1}{\sigma}+\sigma^{\frac{2}{\gamma_2-1}}s_-^{-\frac{2}{\gamma_2-1}}V_2(\frac{|x-q|}{s_-}) \ \ \ \ \ & |x-q|\le s_-,\\
		\frac{1}{2\pi}\ln\frac{1}{\sigma}\cdot\frac{\ln |x-q|}{\ln s_-} & |x-q|>s_-,
	\end{array}
	\right.
\end{equation*}
where $V_2$ is solution to \eqref{5-2} with exponent $\gamma=\gamma_2$. To make $V_{1,\varepsilon},V_{2,\varepsilon}(x)\in C^1$, we need to choose $s_+, s_-$ such that
\begin{equation*}
	\varepsilon^{\frac{2}{\gamma_1-1}}s_+^{-\frac{2}{\gamma_1-1}}|V_1'(1)|=\frac{1}{2\pi}\frac{|\ln \varepsilon|}{|\ln s_+|}, \ \ \ \sigma^{\frac{2}{\gamma_2-1}}s_-^{-\frac{2}{\gamma_2-1}}|V_2'(1)|=\frac{1}{2\pi}\frac{|\ln \sigma|}{|\ln s_-|},
\end{equation*}
which is equivalent to
\begin{equation*}
	s_+=\mu_+\varepsilon, \ \ \ \ \ s_-=\mu_-\sigma
\end{equation*}
for some constants $\mu_+,\mu_->0$. Then a suitable approximate solution to \eqref{5-1} is
\begin{equation}\label{5-3}
	\Psi_\varepsilon(x)=\sum\limits_{k\in \mathbb Z}V_{1,\varepsilon}(x+kl\boldsymbol e_2)-\sum\limits_{k\in\mathbb Z}V_{2,\varepsilon}(x+kl\boldsymbol e_2)
\end{equation}
for $x\in \mathbb R\times (-l/2,l/2)$. Similar to \eqref{2-5}, to make \eqref{5-3} convergent, we assume the sum is understood in the sense
\begin{equation*}
	\begin{split}
		\Psi_\varepsilon&(x)=V_{1,\varepsilon}(x)-V_{2,\varepsilon}(x)\\
		&+\lim\limits_{N\to\infty}\sum\limits_{k=1}^N\left(\sum\limits_{m=\pm k}V_{1,\varepsilon}(x+ml\boldsymbol e_2)-\sum\limits_{m=\pm k}V_{2,\varepsilon}(x+ml\boldsymbol e_2)\right).
	\end{split}
\end{equation*} 
From \eqref{5-1}, for $x\in B_r(p)$ we have
\begin{equation*}
	\begin{split}
		&-\Delta\Psi_\varepsilon-\varepsilon^{-2}\left(\Psi_\varepsilon+W_\varepsilon x_1-\frac{\lambda_+}{2\pi}\ln\frac{1}{\varepsilon}\right)_+^{\gamma_1}\chi_{B_r(p)}\\
		& \ \ \ \ \ \ \ \ \ \ \ \ +\sigma^{-2}\left(-\Psi_\varepsilon-W_\varepsilon x_1-\frac{\lambda_-}{2\pi}\ln\frac{1}{\sigma}\right)_+^{\gamma_2}\chi_{B_r(q)}\\
		&=\varepsilon^{-2}\bigg( \left(V_{1,\varepsilon}(x)-\frac{1}{2\pi}\ln\frac{1}{\varepsilon}\right)_+^{\gamma_1}\\
		& \ \ \ \ \ \ \ -\left(\sum\limits_{k\in\mathbb Z}V_{1,\varepsilon}(x+kl\boldsymbol e_2)-\sum\limits_{k\in\mathbb Z}V_{2,\varepsilon}(x+kl\boldsymbol e_2)+W_\varepsilon x_1-\frac{\lambda_+}{2\pi}\ln\frac{1}{\varepsilon}\right)_+^{\gamma_1}\bigg).
	\end{split}
\end{equation*}
Similarly, for $x\in B_r(q)$ it holds
\begin{equation*}
	\begin{split}
		&-\Delta\Psi_\varepsilon-\varepsilon^{-2}\left(\Psi_\varepsilon+W_\varepsilon x_1-\frac{\lambda_+}{2\pi}\ln\frac{1}{\varepsilon}\right)_+^{\gamma_1}\chi_{B_r(p)}\\
		& \ \ \ \ \ \ \ \ \ \ \ \ +\sigma^{-2}\left(-\Psi_\varepsilon-W_\varepsilon x_1-\frac{\lambda_-}{2\pi}\ln\frac{1}{\sigma}\right)_+^{\gamma_2}\chi_{B_r(q)}\\
		&=\sigma^{-2}\bigg(-\left(V_{2,\varepsilon}(x)-\frac{1}{2\pi}\ln\frac{1}{\sigma}\right)_+^{\gamma_2} \\
		& \ \ \ \ \ \ \ +\left(\sum\limits_{k\in\mathbb Z}V_{2,\varepsilon}(x+kl\boldsymbol e_2)-\sum\limits_{k\in\mathbb Z}V_{1,\varepsilon}(x+kl\boldsymbol e_2)-W_\varepsilon x_1-\frac{\lambda_-}{2\pi}\ln\frac{1}{\sigma}\right)_+^{\gamma_2}\bigg).
	\end{split}
\end{equation*}
To ensure that $\Psi_\varepsilon(x)$ is a good approximation of the solution to \eqref{5-1}, we choose $\lambda_+$ and $\lambda_-$ such that
\begin{equation}\label{5-4}
	\begin{split}
		\lim\limits_{N\to\infty}\sum\limits_{k=1}^N\bigg(\sum\limits_{m=\pm k}V_{1,\varepsilon}(p+ml\boldsymbol e_2)&-\sum\limits_{m=\pm k}V_{2,\varepsilon}(p+ml\boldsymbol e_2)\bigg)\\
		&-V_{2,\varepsilon}(p)-W_\varepsilon d-\frac{\lambda_+}{2\pi}\ln\frac{1}{\varepsilon}=-\frac{1}{2\pi}\ln\frac{1}{\varepsilon},
	\end{split}
\end{equation}
and
\begin{equation}\label{5-5}
	\begin{split}
		\lim\limits_{N\to\infty}\sum\limits_{k=1}^N\bigg(\sum\limits_{m=\pm k}V_{2,\varepsilon}(q+ml\boldsymbol e_2)&-\sum\limits_{m=\pm k}V_{1,\varepsilon}(q+ml\boldsymbol e_2)\bigg)\\
		&-V_{1,\varepsilon}(q)-W_\varepsilon d-\frac{\lambda_-}{2\pi}\ln\frac{1}{\sigma}=-\frac{1}{2\pi}\ln\frac{1}{\sigma}.
	\end{split}
\end{equation}
Notice that the sums in the above two equalities are convergent, since as $|kl|\to \infty$,
\begin{equation*}
	\sum\limits_{m=\pm k}V_{1,\varepsilon}(p+ml\boldsymbol e_2)-\sum\limits_{m=\pm k}V_{2,\varepsilon}(p+ml\boldsymbol e_2)\thickapprox C|kl|^{-2},
\end{equation*}
and
\begin{equation*}
	\sum\limits_{m=\pm k}V_{2,\varepsilon}(q+ml\boldsymbol e_2)-\sum\limits_{m=\pm k}V_{1,\varepsilon}(q+ml\boldsymbol e_2)\thickapprox C|kl|^{-2}.
\end{equation*}
 By \eqref{5-4} and \eqref{5-5}, for $\lambda_+$ and $\lambda_-$ we have the following asymptotic estimate
\begin{equation*}
	\lambda_+=1+O(\frac{\varepsilon}{|\ln \varepsilon|} ), \ \ \  \lambda_-=1+O(\frac{\sigma}{|\ln\sigma|}).
\end{equation*}
Using Pohozaev identity $\int_{B_1(0)}V^p=2\pi|V'(1)|$, one can easily verify that it holds in the sense of measure that as $\varepsilon\to 0$,
\begin{equation*}
	-\Delta\Psi_\varepsilon(x)\rightharpoonup \boldsymbol\delta_p(x)-\boldsymbol\delta_q(x)  \ \ \ \ \ \text{in} \ \ \mathbb R\times (-l/2,l/2).
\end{equation*}
 The error of the approximation by $\Psi_\varepsilon$ is
\begin{equation}\label{5-6}
	\begin{split}
		&-\Delta\Psi_\varepsilon-\varepsilon^{-2}\left(\Psi_\varepsilon+W_\varepsilon x_1-\frac{\lambda_+}{2\pi}\ln\frac{1}{\varepsilon}\right)_+^{\gamma_1}\chi_{B_r(p)}\\
		& \ \ \ \ \ \ \ +\sigma^{-2}\left(-\Psi_\varepsilon-W_\varepsilon x_1-\frac{\lambda_-}{2\pi}\ln\frac{1}{\sigma}\right)_+^{\gamma_2}\chi_{B_r(q)}\\
		&=O(\varepsilon^{-1})\chi_{B_{L\varepsilon}(p)}+O(\sigma^{-1})\chi_{B_{L\sigma}(q)},
	\end{split}
\end{equation}
where $L>0$ is some large constant.

Similarly to the case $0<s<1$, the desirable $l$-symmetric solutions to \eqref{5-1} has the form
\begin{equation*}
	\psi_\varepsilon(x)=\Psi_\varepsilon(x)+\omega_\varepsilon(x),
\end{equation*}
with $x\in \mathbb{R}\times (-l/2,l/2)$, and $\omega_\varepsilon(x)$ a family of $l$-symmetric perturbation terms. Hence we are going to study the equation for $\omega_\varepsilon(x)$.

\subsection{The linear theory}
The linearized operator of \eqref{5-1} at $\Psi_\varepsilon(x)$ is
\begin{equation}\label{6-1}
	\mathbb L_\varepsilon w=-\Delta w-f_u(x,\Psi_\varepsilon) w \ \ \ \ \  \text{in} \ \ \mathbb R\times (-l/2,l/2),
\end{equation}
where $f(x,u)$ is the (nonlinear) function in the right hand side of \eqref{5-1} with $\psi$ being replaced by $u$. So
\begin{equation}\label{6-2}
	\begin{split}
		f_u(x,\Psi_\varepsilon)&=\varepsilon^{-2}\gamma_1\left(\Psi_\varepsilon+W_\varepsilon x_1-\frac{\lambda_+}{2\pi}\ln\frac{1}{\varepsilon}\right)_+^{\gamma_1-1}\chi_{B_r(p)}\\
		& \ \ \ +\sigma(\varepsilon)^{-2}\gamma_2\left(-\Psi_\varepsilon-W_\varepsilon x_1-\frac{\lambda_-}{2\pi}\ln\frac{1}{\sigma(\varepsilon)}\right)_+^{\gamma_2-1}\chi_{B_r(q)}.
	\end{split}
\end{equation}
Hence we can write \eqref{5-1} as
\begin{equation}\label{6-3}
	\mathbb L_\varepsilon\omega_\varepsilon=-E_\varepsilon+R_\varepsilon(\omega_\varepsilon) \ \ \ \ \ \ \  \text{in} \ \ \mathbb R\times (-l/2,l/2),
\end{equation}
where
\begin{equation*}
	E_\varepsilon=-\Delta\Psi_\varepsilon-f(x,\Psi_\varepsilon)
\end{equation*}
and
\begin{equation*}
	R_\varepsilon(\omega_\varepsilon)=f(x,\Psi_\varepsilon+\omega_\varepsilon)-f(x,\Psi_\varepsilon)-f_u(x,\Psi_\varepsilon)\omega_\varepsilon.
\end{equation*}

 Let
\begin{equation*}
	\tilde V=\left\{
	\begin{array}{lll}
		V(|x|) & \text{if} & |x|\le 1,\\
		|V'(1)|\ln\frac{1}{|x|} & \text{if} & |x|>1,
	\end{array}
	\right.
\end{equation*}
where $V$ is the unique solution of \eqref{5-2}. Then the locally linearized operator for our problem is
\begin{equation*}
	\mathbb L_0 w=-\Delta w-\gamma\tilde V_+^{\gamma-1}w \ \ \ \ \  \text{in} \ \ \mathbb R^2.
\end{equation*}
The following nondegeneracy theorem can be found in \cite{Cao5,FW}:
\begin{theorem}\label{thm6-1}
	If $\varphi$ is in the kernel of $\mathbb L_0$, then $\varphi$ is a linear combination of $\frac{\partial \tilde V}{\partial x_1}$ and $\frac{\partial \tilde V}{\partial x_2}$.
\end{theorem}
Hence the kernel of $\mathbb L_\varepsilon$ is spanned by
\begin{equation*}
	Z_\varepsilon(x)=Z_{1,\varepsilon}(x)-Z_{2,\varepsilon}(x),
\end{equation*}
where
\begin{equation*}
	Z_{1,\varepsilon}(x)=\sum\limits_{k\in \mathbb Z}\partial_{x_1}V_{1,\varepsilon}(x+kl\boldsymbol e_2), \ \ \  Z_{2,\varepsilon}(x)=\sum\limits_{k\in \mathbb Z}\partial_{x_1}V_{2,\varepsilon}(x+kl\boldsymbol e_2).
\end{equation*}
We will study the following projected linear problem:
\begin{equation}\label{6-4}
	\begin{cases}
		\mathbb L_\varepsilon\omega_\varepsilon=h(x)+\alpha_\varepsilon f_u(x,\Psi_\varepsilon) Z_\varepsilon(x) \ \ \ \text{in} \ \ \mathbb R\times (-l/2,l/2),\\
		\int_{\mathbb R\times (-l/2,l/2)} f_u(x,\Psi_\varepsilon) Z_\varepsilon(x) \omega_\varepsilon(x)dx=0,\\
		\omega_\varepsilon(x)\to 0 \ \ \ \text{as} \ \ |x_1|\to\infty,
	\end{cases}
\end{equation}
where we assume $l$-symmetric $h(x)$ satisfies
\begin{equation}\label{6-5}
	supp(h(x))\subset B_{L\varepsilon}(p)\cup B_{L\sigma}(q)
\end{equation}
for some large constant $L>0$. The norms we will use for the case $s=1$ are
\begin{equation*}
	\|\omega_\varepsilon\|_*=\sup\limits_{x\in R\times (-l/2,l/2)} \rho(x)^{-1}|\omega_\varepsilon(x)|,
\end{equation*}
where
\begin{equation*}
	\rho(x)=\left|\ln\frac{2r}{r+|x-p|}-\ln\frac{2r}{r+|x-q|}\right|+\lim\limits_{N\to\infty}\sum\limits_{k=1}^N\frac{1}{|x+kl\boldsymbol e_2|^2}
\end{equation*}
with $r>0$ the small constant given in \eqref{5-1}, and
\begin{equation*}
	\|h\|_{**}=\sup\limits_{x\in \mathbb R_-\times (-l/2,l/2)}\varepsilon^2 |h(x)|+ \sup\limits_{x\in \mathbb R_+\times (-l/2,l/2)}\sigma^2|h(x)|.
\end{equation*}

The following a priori estimate is the counterpart of Lemma \ref{lem3-1} in Section 2.
\begin{lemma}\label{lem6-1}
	Assume that
$h(x)$ is $l$-symmetric, which satisfies \eqref{6-5} and $\|h\|_{**}<\infty$. Then there exists a small $\varepsilon_0>0$ such that for any $\varepsilon\in (0,\varepsilon_0)$ and solution pair $(\omega_\varepsilon, \alpha_\varepsilon)$ to \eqref{3-4}, and for some constant $C>0$
	\begin{equation}\label{6-6}
		\|\omega_\varepsilon\|_*+(\sigma(\varepsilon))^{-1}|\alpha_\varepsilon|\le C\|h\|_{**}.
	\end{equation}
\end{lemma}
\begin{proof}
	Similar to the proof of Lemma \ref{lem3-1}, we first  prove
	\begin{equation}\label{6-7}
		(\sigma(\varepsilon))^{-1}|\alpha_\varepsilon|\le C(\|h\|_{**}+o_\varepsilon(1)\|\omega_\varepsilon\|_*).
	\end{equation}
	The coefficient $\alpha_\varepsilon$ is given by
	\begin{equation*}
		\alpha_\varepsilon\int_{\mathbb R\times (-l/2,l/2)}f_u(x,\Psi_\varepsilon)Z_\varepsilon^2dx=\int_{\mathbb R\times (-l/2,l/2)}Z_\varepsilon\mathbb L_\varepsilon\omega_\varepsilon dx-\int_{\mathbb R\times (-l/2,l/2)}hZ_\varepsilon dx.
	\end{equation*}
	Notice that $s_+=\varepsilon\mu_+$ and $s_-=\sigma\mu_-$. For the left hand side we have
	\begin{equation}\label{6-8}
		\begin{split}
			\int_{\mathbb R\times (-l/2,l/2)}f_u(x,\Psi_\varepsilon&)Z_\varepsilon^2dx=(1+o_\varepsilon(1))\varepsilon^{-2}\int_{\mathbb R\times (-l/2,l/2)}\gamma_1\left(V_{1,\varepsilon}-\frac{1}{2\pi}\ln\frac{1}{\varepsilon}\right)_+^{\gamma_1-1}\left(\frac{\partial V_{1,\varepsilon}}{\partial y^1_1}\right)^2dy^1\\
			& \ \ \ \ +(1+o_\varepsilon(1))\sigma^{-2}\int_{\mathbb R\times (-l/2,l/2)}\gamma_2\left(V_{2,\varepsilon}-\frac{1}{2\pi}\ln\frac{1}{\sigma}\right)_+^{\gamma_2-1}\left(\frac{\partial V_{2,\varepsilon}}{\partial y^2_1}\right)^2dy^2\\
			& \ \ \ \ =c_1(1+o_\varepsilon(1))\varepsilon^{-2}+c_2(1+o_\varepsilon(1))\sigma^{-2},
		\end{split}
	\end{equation}
	where $y^1=\frac{x}{s_+}$, $y^2=\frac{x}{s_-}$, and $c_1,c_2>0$ are some constants. For the right hand side, it holds
	\begin{equation*}
		\begin{split}
			&\int_{\mathbb R\times (-l/2,l/2)}Z_\varepsilon\mathbb (-\Delta\omega_\varepsilon) dx=\int_{\mathbb R\times (-l/2,l/2)}\omega_\varepsilon (-\Delta Z_\varepsilon) dx\\
			&=\int_{\mathbb R\times (-l/2,l/2)}\omega_\varepsilon\left( \varepsilon^{-2}\gamma_1\left(V_{1,\varepsilon}(x)-\frac{1}{2\pi}\ln\frac{1}{\varepsilon}\right)_+^{\gamma_1-1}Z_{1,\varepsilon}-\sigma^{-2}\gamma_2\left(V_{2,\varepsilon}(x)-\frac{1}{2\pi}\ln\frac{1}{\varepsilon}\right)_+^{\gamma_2-1}Z_{2,\varepsilon}\right)dx.
		\end{split}
	\end{equation*}
	For $x_1<0$, we have
	\begin{equation*}
		\begin{split}
			&\left|\varepsilon^{-2}\gamma_1\left(V_{1,\varepsilon}(x)-\frac{1}{2\pi}\ln\frac{1}{\varepsilon}\right)_+^{\gamma_1-1}Z_{1,\varepsilon}-\sigma^{-2}\gamma_2\left(V_{2,\varepsilon}(x)-\frac{1}{2\pi}\ln\frac{1}{\varepsilon}\right)_+^{\gamma_2-1}Z_{2,\varepsilon}-f_u(x,\Psi_\varepsilon)Z_\varepsilon\right|\\
			&=\left|\varepsilon^{-2}\gamma_1\left(V_{1,\varepsilon}(x)-\frac{1}{2\pi}\ln\frac{1}{\varepsilon}\right)_+^{\gamma_1-1}Z_{1,\varepsilon}-\varepsilon^{-2}\gamma_1\left(V_{1,\varepsilon}(x)-\frac{1}{2\pi}\ln\frac{1}{\varepsilon}+O(\varepsilon)\right)_+^{\gamma_1-1}Z_\varepsilon\right|\\
			&\le C\varepsilon^{-3+\min\{\gamma_1-1,1\}}\chi_{B_{L\varepsilon}(p)}.
		\end{split}
	\end{equation*}
	For $x_1>0$, the term is
	\begin{equation*}
		\begin{split}
			&\left|-\sigma^{-2}\gamma_2\left(V_{2,\varepsilon}(x)-\frac{1}{2\pi}\ln\frac{1}{\sigma}\right)_+^{\gamma_2-1}Z_{2,\varepsilon}+\sigma^{-2}\gamma_2\left(V_{2,\varepsilon}(x)-\frac{1}{2\pi}\ln\frac{1}{\sigma}+O(\sigma)\right)_+^{\gamma_2-1}Z_\varepsilon\right|\\
			&\le C\sigma^{-3+\min\{\gamma_2-1,1\}}\chi_{B_{L\sigma}(q)}.
		\end{split}
	\end{equation*}
    Using H\"older's inequality, we have
	\begin{equation}\label{6-9}
		\left|\int_{\mathbb R\times (-l/2,l/2)}Z_\varepsilon\mathbb L_\varepsilon\omega_\varepsilon dx\right|\le o_\varepsilon(1)\cdot\|\omega_\varepsilon\|_*\sigma^{-1},
	\end{equation}
	and
	\begin{equation}\label{6-10}
		\left|\int_{\mathbb R\times (-l/2,l/2)}hZ_\varepsilon dx\right|\le \|h\|_{**}\sigma^{-1}.
	\end{equation}
	Then \eqref{6-7} is a direct consequence of \eqref{6-8} \eqref{6-9} and \eqref{6-10}.
	
	Then we will prove
	\begin{equation}\label{6-11}
		\|\omega_\varepsilon\|_*\le C\|h\|_{**},
	\end{equation}
	which is achieved by deriving a contradiction. We assume that there exists a sequence $\{\varepsilon_n\}$ satisfying $\varepsilon_n\to 0$, and solution pairs $(\omega_{\varepsilon_n}, \alpha_{\varepsilon_n})$ to \eqref{6-4} for some $h_{\varepsilon_n}$, such that
	\begin{equation}\label{6-12}
		\|\omega_{\varepsilon_n}\|_*=1, \ \ \ \|h_{\varepsilon_n}\|_{**}\to 0 \ \ \ \text{as} \ n\to\infty.
	\end{equation}
	We want to show that for any $L>0$, it always holds as $ n\to\infty$,
	\begin{equation}\label{6-13}
		\|\omega_{\varepsilon_n}\|_{L^\infty(B_{L\varepsilon_n}(p))}+\|\omega_{\varepsilon_n}\|_{L^\infty(B_{L\sigma_n}(q))}\to 0,
	\end{equation}
where $\sigma_n=\sigma(\varepsilon_n)$.
	Indeed, if \eqref{6-13} is false, then we can suppose the first term satisfies for some $\Lambda_0>0$

$$\|\omega_{\varepsilon_n}\|_{L^\infty(B_{L\varepsilon_n}(p))}\ge \Lambda_0>0.$$

Let
	\begin{equation*}
		\tilde \omega_{\varepsilon_n}(y)=\mu_+^{\frac{2}{\gamma_1-1}}\omega_{\varepsilon_n}(s_+ y+p),
	\end{equation*}
then in every compact set $\tilde\omega_{\varepsilon_n}(y)$ satisfies
	\begin{equation*}
		\begin{split}
			-\Delta \tilde\omega_{\varepsilon_n}(y)&-\gamma_1(\tilde V_1+O(\varepsilon_n))_+^{\gamma_1-1}\tilde\omega_{\varepsilon_n}(y)+o_{\varepsilon_n}(1)\\
			&=\varepsilon_n^2\mu_+^{\frac{2\gamma_1}{\gamma_1-1}}h_{\varepsilon_n}(s_+ y+p)+\varepsilon_n^{-1}\alpha_{\varepsilon_n}\gamma_1\big(\tilde V_1+O(\varepsilon_n)\big)_+^{\gamma_1-1}\left(\frac{\partial \tilde V_1}{\partial y_1}+o_{\varepsilon_n}(1)\right),
		\end{split}
	\end{equation*}
	which is equivalent to
	\begin{equation*}
		-\Delta \tilde\omega_{\varepsilon_n}(y)-\gamma_1V_1^{\gamma_1-1}\tilde\omega_{\varepsilon_n}(y)+o_{\varepsilon_n}(1)=\mathcal R_n(y),
	\end{equation*}
	where
	\begin{equation*}
		\mathcal R_n(y)=\varepsilon_n^2\mu_+^{\frac{2}{\gamma_1-1}}h_{\varepsilon_n}(s_+ y+p)+o_{\varepsilon_n}(1)\cdot \tilde\omega_{\varepsilon_n}(y)+\varepsilon_n^{-1}\alpha_{\varepsilon_n}\gamma_1\big(\tilde V_1+o_{\varepsilon_n}(1)\big)_+^{\gamma_1-1}\left(\frac{\partial \tilde V_1}{\partial y_1}+o_{\varepsilon_n}(1)\right).
	\end{equation*}
	Since $\varepsilon_n^2\mu_+^{\frac{2}{\gamma_1-1}}h_{\varepsilon_n}(y)\to 0$, and $\varepsilon_n^{-1}\alpha_{\varepsilon_n}\le \sigma^{-1}\alpha_{\varepsilon_n}\le C(\|h\|_{**}+o_{\varepsilon_n}(1)\|\omega_{\varepsilon_n}\|_*)=o_{\varepsilon_n}(1)$ on every compact set by \eqref{6-7} and \eqref{6-12}, we have $\mathcal R_n(y)\to 0$. Let $n\to \infty$, we may assume that $\tilde\omega_{\varepsilon_n}$ converge uniformly on compact sets to a function $\tilde \omega$ satisfying
	\begin{equation}\label{6-14}
		\|\tilde \omega\|_{L^\infty(B_{L\mu_+^{-1}}(0))}\ge \Lambda_0\mu_+^{\frac{2}{\gamma_1-1}}.
	\end{equation}
	However, $\tilde \omega$ is even in $y_2$ direction, and is a solution to
	\begin{equation*}
		-\Delta \tilde\omega(y)-\gamma_1 V_1^{\gamma_1-1}\tilde\omega(y)=0.
	\end{equation*}
	By \eqref{6-4}, it also satisfies the orthogonality condition
	\begin{equation*}
		\int_{\mathbb{R}^2}\gamma_1 V_1^{\gamma_1-1}\tilde\omega \frac{\partial \tilde V_1}{\partial y_1}dy=0.
	\end{equation*}
	According to Theorem \ref{thm6-1}, it must hold $\tilde \omega\equiv 0$, which is a contradiction to \eqref{6-14}. Hence we deduce that $\|\omega_{\varepsilon_n}\|_{L^\infty(B_{L\varepsilon_n}(p_{\varepsilon_n}))}\to 0$. For the second term in \eqref{6-13}, we can use a similar method to prove $\|\omega_{\varepsilon_n}\|_{L^\infty(B_{L\sigma_n}(q_{\varepsilon_n}))}\to 0$.
	
	In view of \eqref{3-4}, $\omega_{\varepsilon_n}$ satisfies
	\begin{equation*}
		-\Delta\omega_{\varepsilon_n}=f_u(x,\Psi_0)\omega_{\varepsilon_n}+h_{\varepsilon_n}+\alpha_{\varepsilon_n}f_u(x,\Psi_{\varepsilon_n})Z_{\varepsilon_n}.
	\end{equation*}
	So we have
	\begin{equation*}
		\omega_{\varepsilon_n}(x)=\int_{\mathbb R\times (-l/2,l/2)}K_1(x-z)\left(f_u(x,\Psi_0)\omega_{\varepsilon_n}(x)+h_{\varepsilon_n}(x)+\alpha_{\varepsilon_n}f_u(x,\Psi_{\varepsilon_n})Z_{\varepsilon_n}(x)\right)dz,
	\end{equation*}
	where
	\begin{equation*}
		K_1(x)=\sum\limits_{k\in \mathbb Z}G_1(x+kl\boldsymbol e_2).
	\end{equation*}
	This implies
	\begin{equation*}
		\rho(x)^{-1}|\omega_{\varepsilon_n}(x)|\le C\left( \|\omega_{\varepsilon_n}\|_{L^\infty(B_{L\varepsilon_n}(p))}+
        \|\omega_{\varepsilon_n}\|_{L^\infty(B_{L\sigma_n}(q))}+\|h_{\varepsilon_n}\|_{**}+\sigma_n^{-1}\alpha_{\varepsilon_n}\right).
	\end{equation*}
	Now, from \eqref{6-7} \eqref{6-12} \eqref{6-13}, we can obtain $\|\omega_{\varepsilon_n}\|_*\to 0$ as $n\to\infty$, which is a contradiction to \eqref{6-12}. Thus \eqref{6-11} is obvious, and \eqref{6-6} is the consequence of \eqref{6-7} and \eqref{6-11}.
\end{proof}

We have the following lemma for the projective problem \eqref{6-4}.
\begin{lemma}\label{lem6-2}
	Assume that $h(x)$ is $l$-symmetric, which satisfies \eqref{6-5} and $\|h\|_{**}<\infty$. Then there exists a small $\varepsilon_0>0$ such that for any $\varepsilon\in (0,\varepsilon_0)$, \eqref{6-4} has a unique solution $\omega_\varepsilon=T_\varepsilon h$, where $T_\varepsilon$ is a linear operator of $h$. Moreover, there exists a constant $C>0$ independent of $\varepsilon$ such that
	\begin{equation*}
		\|\omega_\varepsilon\|_*\le C\|h\|_{**}.
	\end{equation*}
\end{lemma}
\begin{proof}
	Denote the Hilbert space
	\begin{equation*}
		H:=\left\{ g\in \dot H^1(\mathbb R\times (-l/2,l/2)) \ : \   g \ \text{is} \ l\text{-symmetric}, \ \int_{\mathbb R\times (-l/2,l/2)} f_u(x,\Psi_\varepsilon) Z_\varepsilon(x) g(x) dx=0\right\}
	\end{equation*}
	endowed with the inner product
	\begin{equation*}
		[u,g]=\int_{\mathbb R\times (-l/2,l/2)}\nabla u(x)\nabla g(x)dx.
	\end{equation*}
    Then we can use Riesz's representation theorem and Fredholm's alternative to derive the desired conclusion, just as what we did in the proof of Lemma \ref{lem3-2}.
\end{proof}

\subsection{The reduction}
To solve \eqref{5-1}, we will deal with \eqref{6-4} for
\begin{equation}\label{7-1}
	h(x)=-E_\varepsilon+R_\varepsilon(\omega_\varepsilon).
\end{equation}
Using Lemma \ref{lem6-2} and a contraction mapping argument, the existence and uniqueness for solutions to \eqref{6-4} and \eqref{7-1} can be established in the following lemma. Since it is similar to that of Lemma \ref{lem4-1}, we omit its proof.
\begin{lemma}\label{lem7-1}
	There are $\varepsilon_0>0$ and $r_0>0$ such that for $\varepsilon\in(0,\varepsilon_0)$ there exists a unique solution $\omega_\varepsilon$ to \eqref{6-4} and \eqref{7-1} in the ball $\|\omega_\varepsilon\|_*\le r_0$. Moreover, it holds
	\begin{equation}\label{7-2}
		\|\omega_\varepsilon\|_*\le C\varepsilon
	\end{equation}
	for some constant $C>0$, and $\omega_\varepsilon$ is continuous with respect to $\varepsilon$.
\end{lemma}

We can easily verify that $\psi_\varepsilon(x)=\Psi_\varepsilon(x)+\omega_\varepsilon(x)$ satisfies
\begin{equation}\label{7-3}
		-\Delta\psi_\varepsilon=f(x,\psi_\varepsilon)+\alpha_\varepsilon f_u(x,\Psi_\varepsilon) Z_\varepsilon(x) \ \ \ \ \  \text{in} \ \ \mathbb R\times (-l/2,l/2).
\end{equation}
To eliminate the second term on the right hand side, we multiply \eqref{7-3} by $Z_\varepsilon$ and integrate over $\mathbb R\times (-l/2,l/2)$ to get
\begin{equation*}
	\alpha_\varepsilon\int_{\mathbb R\times (-l/2,l/2)}f_u(x,\Psi_\varepsilon)Z_\varepsilon^2dx=\int_{\mathbb R\times (-l/2,l/2)}\left(-\Delta\psi_\varepsilon-f(x,\psi_\varepsilon)\right)Z_\varepsilon dx.
\end{equation*}
The following lemma gives a precise condition to ensure that $\alpha_\varepsilon=0$.
\begin{lemma}\label{lem7-2}
	It holds
	\begin{equation*}
		\int_{\mathbb R\times (-l/2,l/2)}\left(-\Delta\psi_\varepsilon-f(x,\psi_\varepsilon)\right)Z_\varepsilon dx=C\left(\frac{1}{2\pi}\lim\limits_{N\to\infty}\sum\limits_{|k|\le N}\frac{p_1-q_1}{|q-p+kl\boldsymbol e_2|^2}+W_\varepsilon\right)+o_\varepsilon(1),
	\end{equation*}
	where $C>0$ is a constant independent of $\varepsilon$, and $p_1,q_1$ denote the first coordinates of $p,q$.
respectively.
\end{lemma}
\begin{proof}
	By \eqref{6-4} and \eqref{7-1}, we have
	\begin{equation*}
		-\Delta\psi_\varepsilon-f(x,\psi_\varepsilon)=\mathbb{L}_\varepsilon\psi_\varepsilon+E_\varepsilon-R_\varepsilon(\omega_\varepsilon)  \ \ \ \ \ \ \ \ \ \text{in} \ \ \mathbb R\times (-l/2,l/2).
	\end{equation*}
	Multiplying this equality by $Z_\varepsilon$ and integrating over $\mathbb R\times (-l/2,l/2)$, we obtain
	\begin{equation*}
		\begin{split}
			\int_{\mathbb R\times (-l/2,l/2)}\big(-\Delta\psi_\varepsilon&-f(x,\psi_\varepsilon)\big)Z_\varepsilon dx\\
			&=\int_{\mathbb R\times (-l/2,l/2)}\mathbb{L}_\varepsilon\omega_\varepsilon Z_\varepsilon dx+\int_{\mathbb R\times (-l/2,l/2)}E_\varepsilon Z_\varepsilon dx-\int_{\mathbb R\times (-l/2,l/2)}R_\varepsilon(\omega_\varepsilon)Z_\varepsilon dx.
		\end{split}
	\end{equation*}
	
	By the choice of $\lambda_+$ and $\lambda_-$ in \eqref{5-4} and \eqref{5-5}, $E_\varepsilon$ can be split into
	\begin{equation*}
		E_\varepsilon=E_{1,\varepsilon}+E_{2,\varepsilon},
	\end{equation*}
	where
	\begin{equation*}
		\begin{split}
			E_{1,\varepsilon}&=\varepsilon^{-2}\chi_{B_{L\varepsilon}(p)}\bigg(\left(V_{1,\varepsilon}(x)-\frac{1}{2\pi}\ln\frac{1}{\varepsilon}\right)^{\gamma_1}_+-\bigg(V_{1,\varepsilon}(x)-\frac{1}{2\pi}\ln\frac{1}{\varepsilon}\\
			&+\sum\limits_{k\neq0}V_{1,\varepsilon}(x+kl\boldsymbol e_2)-\sum\limits_{k\in\mathbb Z}V_{2,\varepsilon}(x+kl\boldsymbol e_2)-\sum\limits_{k\neq0}V_{1,\varepsilon}(p+kl\boldsymbol e_2)\\
			&+\sum\limits_{k\in\mathbb Z}V_{2,\varepsilon}(p+kl\boldsymbol e_2)
			+W_\varepsilon(x_1-d)\bigg)^{\gamma_1}_+\bigg),
		\end{split}
	\end{equation*}
	and
	\begin{equation*}
		\begin{split}
			E_{2,\varepsilon}&=\sigma^{-2}\chi_{B_{L\sigma}(q)}\bigg(-\bigg(V_{2,\varepsilon}(x)-\frac{1}{2\pi}\ln\frac{1}{\sigma}\bigg)^{\gamma_2}_++\bigg(V_{2,\varepsilon}(x)-\frac{1}{2\pi}\ln\frac{1}{\sigma}\\
			&+\sum\limits_{k\neq0}V_{2,\varepsilon}(x+kl\boldsymbol e_2)-\sum\limits_{k\in\mathbb Z}V_{1,\varepsilon}(x+kl\boldsymbol e_2)-\sum\limits_{k\neq0}V_{2,\varepsilon}(q+kl\boldsymbol e_2)\\
			&+\sum\limits_{k\in\mathbb Z}V_{1,\varepsilon}(q+kl\boldsymbol e_2)-W_\varepsilon(x_1-d)\bigg)^{\gamma_2}_+\bigg).
		\end{split}
	\end{equation*}	
	Since
	\begin{equation*}
		|E_{1,\varepsilon}|\le C\varepsilon^{-2}\cdot\varepsilon \ \ \  \text{and} \ \ \ |E_{2,\varepsilon}|\le C\sigma^{-2}\cdot\sigma,
	\end{equation*}
	It holds
	\begin{equation*}
		\int_{\mathbb R\times (-l/2,l/2)}E_{1,\varepsilon} Z_{2,\varepsilon} dx\le C\varepsilon, \ \ \  \int_{\mathbb R\times (-l/2,l/2)}E_{2,\varepsilon} Z_{1,\varepsilon} dx\le C\sigma.
	\end{equation*}
	By Taylor's formula, we have
	\begin{equation*}
		\begin{split}
			E_{1,\varepsilon}=&-\varepsilon^{-2}\chi_{B_{L\varepsilon}(p)}\left(V_{1,\varepsilon}(x)-\frac{1}{2\pi}\ln\frac{1}{\varepsilon}\right)^{\gamma_1-1}_+\bigg(\sum\limits_{k\neq0}V_{1,\varepsilon}(x+kl\boldsymbol e_2)\\
			&-\sum\limits_{k\in\mathbb Z}V_{2,\varepsilon}(x+kl\boldsymbol e_2)-\sum\limits_{k\neq0}V_{1,\varepsilon}(p+kl\mathbf e_2)+\sum\limits_{k\in\mathbb Z}V_{2,\varepsilon}(p+kl\boldsymbol e_2)\\
			&+W_\varepsilon(x_1-d)\bigg) +\varepsilon^{-2}\cdot\varepsilon^{\gamma_1-1}\chi_{B_{L\varepsilon}(p)}.
		\end{split}
	\end{equation*}
	We then integrate by parts and use the asymptotic behavior of $V_{1,\varepsilon}, V_{2,\varepsilon}$ to obtain
	\begin{equation*}
		\begin{split}
			\int_{\mathbb R\times (-l/2,l/2)}&E_{1,\varepsilon} Z_{1,\varepsilon} dx=\varepsilon^{-2}\int_{\mathbb R\times (-l/2,l/2)}\left(V_{1,\varepsilon}(x)-\frac{1}{2\pi}\ln\frac{1}{\varepsilon}\right)^{\gamma_1}_+\\
			& \ \ \ \ \ \ \ \ \ \ \left(\sum\limits_{k\neq0}\partial_{x_1}V_{1,\varepsilon}(x+kl\boldsymbol e_2)
			-\sum\limits_{k\in\mathbb Z}\partial_{x_1}V_{2,\varepsilon}(x+kl\boldsymbol e_2)+W_\varepsilon+O(\varepsilon)\right)dx\\
			&=\varepsilon^{-2}\int_{\mathbb R\times (-l/2,l/2)}\left(V_{1,\varepsilon}(x)-\frac{1}{2\pi}\ln\frac{1}{\varepsilon}\right)^{\gamma_1}_+\\
			& \ \ \ \ \ \ \ \ \ \ \left(\frac{1}{2\pi}\lim\limits_{N\to\infty}\sum\limits_{|k|\le N}\frac{p_1-q_1}{|p-q+kl\boldsymbol e_2|^2}+W_\varepsilon+O(\varepsilon)\right)dx\\
			&=C_1\left(\frac{1}{2\pi}\lim\limits_{N\to\infty}\sum\limits_{|k|\le N}\frac{p_1-q_1}{|p-q+kl\boldsymbol e_2|^2}+W_\varepsilon\right)+O(\varepsilon),	
		\end{split}
	\end{equation*}
	where $p_1,q_1$ denotes the first coordinates of $p,q$. Similarly, we can obtain
	\begin{equation*}
		\begin{split}
			-\int_{\mathbb R\times (-l/2,l/2)}E_{2,\varepsilon} Z_{2,\varepsilon} dx=C_2\left(\frac{1}{2\pi}\lim\limits_{N\to\infty}\sum\limits_{|k|\le N}\frac{-q_1+p_1}{|q-p+kl\boldsymbol e_2|^2}+W_\varepsilon\right)+O(\sigma).		
		\end{split}
	\end{equation*}
	Hence we can derive
	\begin{equation}\label{7-4}
		\begin{split}
			\int_{\mathbb R\times (-l/2,l/2)}E_\varepsilon Z_\varepsilon dx=C\left(\frac{1}{2\pi}\lim\limits_{N\to\infty}\sum\limits_{|k|\le N}\frac{p_1-q_1}{|q-p+kl\boldsymbol e_2|^2}+W_\varepsilon\right)+o_\varepsilon(1).
		\end{split}
	\end{equation}
	
	Recall that $\sigma(\varepsilon)$ satisfies assumption \textbf{(H)} with $\tau=\min\{\gamma_2,2\}$. We also have
	\begin{equation}\label{7-5}
		\begin{split}
			\int_{\mathbb R\times (-l/2,l/2)}R_\varepsilon(\omega_\varepsilon)Z_\varepsilon dx&=\int_{\mathbb R_-\times (-l/2,l/2)}R_\varepsilon(\omega_\varepsilon)Z_\varepsilon dx+\int_{\mathbb R_+\times (-l/2,l/2)}R_\varepsilon(\omega_\varepsilon)Z_\varepsilon dx\\
			&\le C\|\omega_\varepsilon\|_*^{\min\{\gamma_1,2\}}\varepsilon^{-1}+C\|\omega_\varepsilon\|_*^{\min\{\gamma_2,2\}}\sigma^{-1}\\
			&\le C\varepsilon^{\min\{\gamma_1-1, 1\}}+C\varepsilon^{\min\{\gamma_2, 2\}}\sigma^{-1}=o_\varepsilon(1),
		\end{split}
	\end{equation}
    and
	\begin{equation}\label{7-6}
		\begin{split}
			\int_{\mathbb R\times (-l/2,l/2)}\mathbb{L}_\varepsilon\omega_\varepsilon Z_\varepsilon dx&=\int_{\mathbb R_-\times (-l/2,l/2)}\mathbb{L}_\varepsilon\omega_\varepsilon Z_\varepsilon dx+\int_{\mathbb R_+\times (-l/2,l/2)}\mathbb{L}_\varepsilon\omega_\varepsilon Z_\varepsilon dx\\
			&\le C\|\omega_\varepsilon\|_*\varepsilon^{\min\{\gamma_1-2,0\}}+C\|\omega_\varepsilon\|_*\sigma^{\min\{\gamma_2-2,0\}}\\
			&\le C\varepsilon^{\min\{\gamma_1-1,1\}}+C\varepsilon\cdot\sigma^{\min\{\gamma_2-2,0\}}=o_\varepsilon(1).
		\end{split}
	\end{equation}
	Using \eqref{7-4} \eqref{7-5} and \eqref{7-6}, we arrive at the conclusion.	

\end{proof}

Having made all the necessary preparation, now we are in a position to prove Theorem \ref{thm1} and  \ref{thm2} for the case $s=1$.

{\bf Proof of Theorem \ref{thm1} and  \ref{thm2} with $s=1$:}
We are to solve $\alpha_\varepsilon=0$ directly. Since it holds
\begin{equation*}
	\int_{\mathbb R\times (-l/2,l/2)}f_u(x,\Psi_\varepsilon)Z_\varepsilon^2dx>0,
\end{equation*}
the condition $\alpha_\varepsilon=0$ is equivalent to
\begin{equation*}
	W_\varepsilon=\frac{1}{2\pi}\lim\limits_{N\to\infty}\sum\limits_{|k|\le N}\frac{-p_1+q_1}{|q-p+kl\boldsymbol e_2|^2}+o_\varepsilon(1)
\end{equation*}
by Lemma \ref{lem7-2}. Hence if $a=0$, $W_\varepsilon$ should satisfy
\begin{equation*}
	W_\varepsilon=\frac{1}{2\pi}\lim\limits_{N\to\infty}\sum\limits_{|k|\le N}\frac{2d}{4d^2+k^2l^2}+o_\varepsilon(1);
\end{equation*}
while for $a=l/4$, the condition turns to be
\begin{equation*}
	W_\varepsilon=\frac{1}{2\pi}\lim\limits_{N\to\infty}\sum\limits_{|k|\le N}\frac{2d}{4{d}^2+(kl+\frac{l}{2})^2}+o_\varepsilon(1).
\end{equation*}
The weak convergence of $\vartheta_{0,\varepsilon}$ is obvious as $\varepsilon\to 0$. By the choice of $p,q$ we get the existence of $W_\varepsilon$ so that
the corresponding $\alpha_\varepsilon=0$.  $C^1$ property follows from regularity theory for Laplacian. Thus the proof is complete.
\qed

\section{General result for $a\in (0,l/2)$}
In this section, we will investigate the general case $a\in(0,l/2)$. Due to the loss of symmetry with respect to $x_1$-direction, the travelling speed $\mathbf U_\varepsilon$ is no longer $x_2$-directional when $a\neq 0$ and $a\neq l/4$ hold simultaneously. For simplicity, we still denote the $l$-periodicity restricted in $\mathbb{R}\times (-l,l)$ as $l$-symmetry, and make the following decomposition for the uniform travelling speed
\begin{equation}\label{8-1}
	\mathbf U_\varepsilon=\overline W_\varepsilon \boldsymbol e_1+W_\varepsilon \boldsymbol e_2.
\end{equation}
Then for $0<s<1$, the aim is to find a family of $l$-symmetry solutions to
\begin{equation}\label{8-2}
	\begin{cases}
		(-\Delta)_*^s\psi=\varepsilon^{(2-2s)\gamma_1-2}(\psi+W_\varepsilon x_1-\overline W_\varepsilon x_2-\varepsilon^{2s-2}\lambda_+)_+^{\gamma_1}\chi_{B_r(p)}\\
		\,\,\, \ \ \ \ \ \ \ \ -\sigma^{(2-2s)\gamma_2-2}(-\psi-W_\varepsilon x_1+\overline W_\varepsilon x_2-\sigma^{2s-2}\lambda_-)_+^{\gamma_2}\chi_{B_r(q)} \ \ \ \ \ \text{in} \ \ \mathbb R\times (-l/2,l/2),\\
		\psi(x)\to 0 \ \ \ \text{as} \ \  |x_1|\to \infty.
	\end{cases}
\end{equation}
For $s=1$, the equation comes to be
\begin{equation}\label{8-3}
	\begin{cases}
		-\Delta\psi=\varepsilon^{-2}\left(\psi+W_\varepsilon x_1-\overline W_\varepsilon x_2 -\frac{\lambda_+}{2\pi}\ln\frac{1}{\varepsilon}\right)_+^{\gamma_1}\chi_{B_r(p)}\\
		\ \ \ \ \ \ \ \ \ \ \ -\sigma^{-2}\left(-\psi-W_\varepsilon x_1+\overline W_\varepsilon x_2-\frac{\lambda_-}{2\pi}\ln\frac{1}{\sigma}\right)_+^{\gamma_2}\chi_{B_r(q)} \ \ \ \ \ \ \ \ \ \ \text{in} \ \ \mathbb R\times (-l/2,l/2),\\
		\psi(x)\to 0 \ \ \ \text{as} \ \  |x_1|\to \infty.
	\end{cases}
\end{equation}
Here, $p$ and $q$ are also given by
\begin{equation*}
	p=(-d, -a), \ \ \  q=(d, a).
\end{equation*}
However, different from our assumption in Section 2 and 3, the parameter $a$ can be chosen in $(0,l/2)$ arbitrarily in this situation, and $d$ can take any nonnegative values.

Readers can easily verify that the construction of approximate solutions is almost the same as before, and an $l$-symmetric solution to \eqref{8-2} or \eqref{8-3} has the form
\begin{equation*}
	\psi_\varepsilon(x)=\Psi_\varepsilon(x)+\omega_\varepsilon(x),
\end{equation*}
with $x\in \mathbb{R}\times (-l/2,l/2)$, and $\omega_\varepsilon(x)$ an $l$-symmetric error term. Using the same notations as in previous sections, we are to consider the equation
\begin{equation}\label{8-4}
	\mathbb L_\varepsilon\omega_\varepsilon=-E_\varepsilon+R_\varepsilon(\omega_\varepsilon) \ \ \ \ \ \ \  \text{in} \ \ \mathbb R\times (-l/2,l/2),
\end{equation}
where the definition of $\mathbb L_\varepsilon$, $E_\varepsilon$ and $R_\varepsilon$ are given in Section 2 and 3. Compared with the case $a=0$ or $l/4$, the essential difference for general $a\in (0,l/2)$ on linear theory is: The kernel of $\mathbb L_\varepsilon$ is $2$-dimensional due to the loss of symmetry. Besides $Z_\varepsilon$ introduced before, there exists another basis of the kernel. When $0<s<1$, it is
\begin{equation*}
	\overline Z_\varepsilon=\varepsilon^{2s-2}\sum\limits_{k\in \mathbb Z}\partial_{x_2}U_{1,\varepsilon}(x+kl\boldsymbol e_2)-\sigma^{2s-2}\sum\limits_{k\in \mathbb Z}\partial_{x_2}U_{2,\varepsilon}(x+kl\boldsymbol e_2);
\end{equation*}
while for $s=1$, it turns to be
\begin{equation*}
	\overline Z_\varepsilon=\sum\limits_{k\in \mathbb Z}\partial_{x_2}V_{1,\varepsilon}(x+kl\boldsymbol e_2)-\sum\limits_{k\in \mathbb Z}\partial_{x_2}V_{2,\varepsilon}(x+kl\boldsymbol e_2).
\end{equation*}
Recall the nondegeneracy results in Theorem \ref{thm3-1} and Theorem \ref{thm6-1}. Using the radial symmetry of ground states, $Z_\varepsilon$ and $\overline Z_\varepsilon$ satisfy the approximate orthogonal condition
\begin{equation}\label{8-5}
	\int_{\mathbb R\times (-l/2,l/2)} f_u(x,\Psi_\varepsilon) Z_\varepsilon(x) 	\overline Z_\varepsilon(x) dx=O_\varepsilon(1)\cdot (\sigma(\varepsilon))^{-1}
\end{equation}
Hence in general, the projected linear problem is
\begin{equation}\label{8-6}
	\begin{cases}
		\mathbb L_\varepsilon\omega_\varepsilon=h(x)+\alpha_{1,\varepsilon} f_u(x,\Psi_\varepsilon) Z_\varepsilon(x)+\alpha_{2,\varepsilon} f_u(x,\Psi_\varepsilon) \overline Z_\varepsilon(x) \ \ \ \ \ \text{in} \ \ \mathbb R\times (-l/2,l/2)\\
		\omega_\varepsilon(x)\to 0 \ \ \ \text{as} \ \ |x_1|\to\infty.
	\end{cases}
\end{equation}
together with orthogonal conditions
\begin{equation}\label{8-7}
	\int_{\mathbb R\times (-l/2,l/2)} f_u(x,\Psi_\varepsilon) Z_\varepsilon(x) \omega_\varepsilon(x)dx=0 \ \ \ \text{and} \ \ \ 	\int_{\mathbb R\times (-l/2,l/2)} f_u(x,\Psi_\varepsilon) \overline Z_\varepsilon(x) \omega_\varepsilon(x)dx=0.
\end{equation}
Employing the approximate orthogonal condition \eqref{8-5}, we can derive the linear theory for \eqref{8-6} and \eqref{8-7} just as we did in previous two sections. Then, using a contraction mapping argument, a family of solutions $\omega_\varepsilon$ to \eqref{8-6} and \eqref{8-7} is obtained, which satisfies the estimate
\begin{equation*}
	\|\omega_\varepsilon\|_*\le C\varepsilon^{3-2s}.
\end{equation*}
We leave out the details of proof for readers.

By the discussion given in Section 2 and 3, we only need to find suitable conditions such that $\alpha_{1,\varepsilon}=0$ and $\alpha_{2,\varepsilon}=0$ hold simultaneously. In view of \eqref{8-6}, we have
\begin{equation*}
	(-\Delta)_*^s\psi_\varepsilon=f(x,\psi_\varepsilon)+\alpha_{1,\varepsilon} f_u(x,\Psi_\varepsilon) Z_\varepsilon(x)+\alpha_{2,\varepsilon} f_u(x,\Psi_\varepsilon) \overline Z_\varepsilon(x) \ \ \ \ \  \text{in} \ \ \mathbb R\times (-l/2,l/2),
\end{equation*}
where $(-\Delta)_*^s$ equals $-\Delta$ when $s=1$. To apply energy method, we multiply the above equality by $Z_\varepsilon(x)$, $\overline Z_\varepsilon(x)$ separately. Using the approximate orthogonal condition \eqref{8-5} once more, and noticing that $\alpha_{i,\varepsilon}\le \sigma(\varepsilon)\cdot \|h\|_{**}\le \sigma(\varepsilon)\cdot\varepsilon^{3-2s}$, we deduce
\begin{equation*}
	\alpha_{1,\varepsilon}\int_{\mathbb R\times (-l/2,l/2)}f_u(x,\Psi_\varepsilon)Z_\varepsilon^2dx+o_\varepsilon(1)=\int_{\mathbb R\times (-l/2,l/2)}\left((-\Delta)_*^{s}\psi_\varepsilon-f(x,\psi_\varepsilon)\right)Z_\varepsilon dx.
\end{equation*}
and
\begin{equation*}
	\alpha_{2,\varepsilon}\int_{\mathbb R\times (-l/2,l/2)}f_u(x,\Psi_\varepsilon)\overline Z_\varepsilon^2dx+o_\varepsilon(1)=\int_{\mathbb R\times (-l/2,l/2)}\left((-\Delta)_*^{s}\psi_\varepsilon-f(x,\psi_\varepsilon)\right)\overline Z_\varepsilon dx.
\end{equation*}
Then, by a similar procedure as in the proofs of Lemma \ref{lem4-2} and \ref{lem7-2}, we can derive that
\begin{equation*}
	\int_{\mathbb R\times (-l/2,l/2)}\left((-\Delta)_*^s\psi_\varepsilon-f(x,\psi_\varepsilon)\right)Z_\varepsilon dx=C\left(\mathcal C_s\lim\limits_{N\to\infty}\sum\limits_{|k|\le N}\frac{p_1-q_1}{|q-p+kl\boldsymbol e_2|^{4-2s}}+W_\varepsilon\right)+o_\varepsilon(1),
\end{equation*}
and
\begin{equation*}
	\int_{\mathbb R\times (-l/2,l/2)}\left((-\Delta)_*^s\psi_\varepsilon-f(x,\psi_\varepsilon)\right)Z_\varepsilon dx=C\left(\mathcal C_s\lim\limits_{N\to\infty}\sum\limits_{|k|\le N}\frac{-p_2+q_2}{|q-p+kl\boldsymbol e_2|^{4-2s}}+\overline W_\varepsilon\right)+o_\varepsilon(1),
\end{equation*}
where $p_i,q_i$ denote the $i$-th coordinate of $p,q$ respectively for $i=1,2$. Thus conditions $\alpha_{1,\varepsilon}=0$ and $\alpha_{2,\varepsilon}=0$ turn to be
\begin{equation}\label{8-8}
	W_\varepsilon=\mathcal C_s\lim\limits_{N\to\infty}\sum\limits_{|k|\le N}\frac{2d}{(4d^2+(kl+2a)^2)^{2-s}}+o_\varepsilon(1),
\end{equation}
and
\begin{equation}\label{8-9}
	\overline W_\varepsilon=\mathcal C_s\lim\limits_{N\to\infty}\sum\limits_{|k|\le N}\frac{kl+2a}{(4{d}^2+(kl+2a)^2)^{2-s}}+o_\varepsilon(1).
\end{equation}
Combining \eqref{8-8} and \eqref{8-9}, we see that the reduction conditions are equivalent to
\begin{equation*}
	\mathbf U_\varepsilon=-\mathcal C_s\lim\limits_{N\to\infty}\sum\limits_{|k|\le N}\frac{(p-q+kl\boldsymbol e_2)^\perp}{|p-q+kl\boldsymbol e_2|^{4-2s}}+o_\varepsilon(1).
\end{equation*}
The weak convergence of $\psi_\varepsilon$ is obvious, and $C^1$ regularity is from a standard theory. So we have actually proved Theorem \ref{thm3}.

\phantom{s}
\thispagestyle{empty}

\end{document}